\documentclass[11pt]{amsart}
\usepackage[utf8]{inputenc}
\usepackage{amsmath,amsthm,amsfonts,amssymb,amscd,mathtools}
\usepackage{verbatim}
\usepackage{mathrsfs}
\usepackage{multicol}
\usepackage{tabularx}
\usepackage[usenames,dvipsnames]{color}
\usepackage{enumerate}
\usepackage{graphicx}
\usepackage{color,xcolor}
\usepackage{bbm}
\usepackage[all]{xy}
\usepackage{hyperref}
\usepackage[capitalise]{cleveref}
\usepackage{pb-diagram,tikz-cd}

\theoremstyle{definition}
\newtheorem{thm}{Theorem}[subsection]
\newtheorem{prop}[thm]{Proposition}
\newtheorem{cor}[thm]{Corollary}
\newtheorem{con}[thm]{Conjecture}
\newtheorem{lem}[thm]{Lemma}

\newtheorem{ex}[thm]{Example}

\numberwithin{equation}{subsection}

\definecolor{mainCol}{RGB}{127, 49, 79} 
\definecolor{TextCol}{rgb}{1.00,1.00,1.00} 
\definecolor{BoxCol}{RGB}{50,31,19}
\definecolor{DarkSlateBlue}{RGB}{72,61,139}
\definecolor{DarkSlateGray}{RGB}{47,79,79}
\definecolor{Aquamarine1}{RGB}{127,255,212}
\definecolor{BlueViolet}{RGB}{138,43,226}
\definecolor{AquaViolet}{RGB}{83,107,241}
\definecolor{Aquamarine1Dual}{RGB}{255,127,170}
\definecolor{Black}{RGB}{0,0,0}
\definecolor{A}{RGB}{44,245,135}
\definecolor{SpringGreen}{RGB}{0,255,128}
\definecolor{SpringGreen1}{RGB}{128,0,255}
\definecolor{SpringGreen2}{RGB}{255,128,0}
\definecolor{Clayton}{RGB}{92,148,197}
\definecolor{amarelo}{RGB}{255,255,0}
\definecolor{VerdeCiano}{RGB}{0,255,161}
\definecolor{Ciano}{RGB}{0,255,255}

\def\ord#1^#2{#1$^{\text{#2}}$}
\def\lie#1{\mathfrak{#1}}
\def\tlie#1{\tilde{\mathfrak{#1}}}
\def\hlie#1{\hat{\mathfrak{#1}}}

\def\uqr#1^#2{\text{$U_q^{#2}(\lie #1)$}}

\def\uqhr#1^#2{\text{$U_q^{#2}(\hlie #1)$}}
\def\us#1^#2{\text{$U_{\xi}^{#2}(\lie #1)$}}
\def\ush#1^#2{\text{$U_{\xi}^{#2}(\hlie #1)$}}
\def\dus#1^#2{\text{$\dot{U}_{\xi}^{#2}(\lie #1)$}}
\def\dush#1^#2{\text{$\dot{U}_{\xi}^{#2}(\hlie #1)$}}

\def\wtl{{\rm wt}_\ell}
\def\wt{{\rm wt}}

\def\supp{{\rm supp}}

\def\ch{{\rm ch}}
\def\qch{{\rm qch}}

\def\opl_#1^#2{\text{\scriptsize$\bigoplus\limits_{\text{\normalsize$#1$}}^{\text{\normalsize$#2$}}$}}
\def\otm_#1^#2{\text{\scriptsize$\bigotimes\limits_{\text{\footnotesize$#1$}}^{\text{\footnotesize$#2$}}$}}

\def\bs#1{\boldsymbol{#1}}
\def\endd{\hfill$\diamond$}
\def\ffbox#1{\setbox9=\hbox{$\scriptstyle\overline{1}$}\framebox[0.55cm][c]{\rule{0mm}{\ht9}${\scriptstyle #1}$}}

\begin{document}

\title[Three-vertex prime graphs and reality of trees]{Three-vertex prime graphs and reality of trees}
\author[A. Moura and C. Silva]{Adriano Moura and Clayton  Silva}

\address{Departamento de Matemática, Universidade Estadual de Campinas, Campinas - SP - Brazil, 13083-859.}
\email{aamoura@ime.unicamp.br}
\email{ccris22@gmail.com}

\thanks{This work was developed as part of the Ph.D. project of the second author, which was supported by a PICME grant.
	The work of the first author was partially supported by CNPq grant 304261/2017-3 and Fapesp grant 2018/23690-6. }

\begin{abstract}
	We continue the study of prime simple modules for quantum affine algebras from the perspective of $q$-fatorization graphs.
	In this paper we establish several properties related to simple modules whose  $q$-factorization graphs are afforded by trees. The two most important of them are proved for type $A$. The first completes the classification of the prime simple modules with three $q$-factors by giving a precise criterion for the primality of a $3$-vertex line which is not totally ordered. Using a very special case of this criterion, we then show that a simple module whose $q$-factorization graph is afforded by an arbitrary tree  is real. Indeed, the proof of the latter works for all types, provided the aforementioned special case is settled in general.
\end{abstract}

\maketitle

\section{Introduction}

This paper is a continuation of \cite{mosi}, where we introduced the combinatorial notion of a $q$-fatorization graph intended as a tool to study and express results related to the  classification of prime simple  modules for quantum affine algebras. While the main result of \cite{mosi} focused on graphs whose vertex sets are totally ordered, the present paper goes in the opposite direction. Most of this paper is dedicated to proving a first step towards describing a set of necessary and sufficient conditions for the primality of simple modules whose  $q$-fatorization graph  is afforded by a tree. The only type of tree which is also totally ordered are the monotinic lines
\begin{equation*}
	\begin{tikzcd}
		\circ \arrow[r] & \circ \arrow[r] & \cdots 
	\end{tikzcd} 
\end{equation*}
which are prime by \cite[Theorem 3.5.4]{mosi}. This first step is \Cref{t:3lineprime}, which completes the proposed goal in the case that the underlying simple Lie algebra is of type $A$ and the $q$-factorization graph has three vertices. Its proof occupies about half of this paper. The present argument can be used to study all other types, but it heavily explores the precise description of the sets $\mathscr R_{i,j}^{r,s}$ which, by definition, encode the answer for graphs with two vertices (their definition is reviewed in \Cref{ss:psm}). Hence, the study of other types using similar arguments requires a case-by-case analysis. 

In the seminal paper \cite{hele:cluster},  Hernandez and Leclerc  showed that the finite-dimensional representation theory of quantum affine algebras and cluster algebras are deeply interconnected. In particular, the notion of cluster variables gave rise to the notion of real prime simple modules. A real module is a simple module whose tensor square is also simple. Thus, beside classifying the prime simple modules, another important task is that of classifying the real ones. As an application of a very particular case of \Cref{t:3lineprime}, we prove \Cref{t:realtree}, that shows that every simple module whose $q$-factorization graph is afforded by a tree is real. This particular case (\Cref{c:3aline}) is expected to be true for all types and  can possibly be proved in a simpler manner with different techniques than those employed here (we have used no cluster algebra argument and made very rudimentary use of qcharacters). If that is done, the present proof of \Cref{t:realtree} works, exactly as written here, for all types.

Beside these two main results, we also prove a few general criteria for deciding the reducibility or simplicity of a tensor product of two simple modules for which the $q$-factorization graph of the product of the two Drinfeld polynomials is a tree,  such as Propositions \ref{p:critsimp} and \ref{c:nonadcjprim}. We also give a few illustrative examples of the usability of the several criteria we have established here as well as in \cite[Section 4]{mosi}. For instance, it was shown in \cite{mosi} that every connected subgraph of a prime tree is also prime. The converse of this is false and, in \Cref{ex:cosubpt=>p},  we give an example of a non-prime tree with four vertices, all of whose proper connected subgraphs are prime. On the other hand, in \Cref{ex:cesubpt=>p}, we give an example of a $q$-factorization graph afforded by a $4$-cycle which is prime but contains connected three-vertex subgraphs which are not prime. Note that we are able to check the primality of the graphs in these examples using our set of criteria, even though they are not covered by the main results from \cite{mosi} nor by the ones from the present paper.

The paper is organized as follows. In \Cref{s:main}, we review the basic terminology and notation, state the main results, and develop the aforementioned examples. In \Cref{s:crit}, we review some further background and prove Propositions \ref{p:critsimp} and \ref{c:nonadcjprim}. \Cref{s:3primes} is dedicated to the proofs of Theorems \ref{t:3lineprime} and \ref{t:realtree}, starting with  a brief review of short exact sequences describing tensor products of fundamental modules for type $A$ in \Cref{soctpkr} (the only part of the paper using qcharacters). \Cref{ss:tpfund} prepares the ground for the proof of \Cref{t:3lineprime}, starting from a broader setup (for all types) by discussing  a few initial steps for obtaining  reducibility criteria for the tensor product of a Kirillov-Reshetikhin  module with a general simple module. Lemmas \ref{l:disconect} and \ref{l:j'inJ} give preliminary answers in this direction. The lengthy \Cref{ss:3linecut} contains the proof of \Cref{t:3lineprime}, while   \Cref{t:realtree} is proved in the final subsection.

\section{Basic Background and Main Results}\label{s:main}

Throughout the paper, let $\mathbb C$ and  $\mathbb Z$ denote the sets of complex numbers and integers, respectively. Let also $\mathbb Z_{\ge m} ,\mathbb Z_{< m}$, etc. denote the obvious subsets of $\mathbb Z$. Given a ring $\mathbb A$, the underlying multiplicative group of units is denoted by $\mathbb A^\times$. 
The symbol $\cong$ means ``isomorphic to''. We shall use the symbol $\diamond$ to mark the end of remarks, examples, and statements of results whose proofs are postponed. The symbol \qedsymbol\ will mark the end of proofs as well as of statements whose proofs are omitted.

\subsection{Quantum Algebras and Their Finite-Dimensional Representations}\label{ss:clalg}
Since this paper should be regarded as a continuation of \cite{mosi}, in order to avoid excessive repetition of background, we kindly ask the reader to refer to Sections 2.3, 2.4, and 3.1 of that paper. Let us fix some additional notation. If $J\subseteq I$, we let $w_0^J$ denote the longest element of the Weyl group of the diagram subalgebra $\lie g_J\subseteq \lie g$, regarded as a subgroup of $\mathcal W$ in the obvious way. Moreover, 
we introduce the following notation which will be often used in the main proof of this paper.
Given $i,j,k\in I$, set
\begin{equation*}
	d_{i,j}^k = \frac{d(k,i)+d(k,j)-d(i,j)}{2}
\end{equation*}
and note
\begin{equation}\label{dikj}
	d_{i,j}^k =  \begin{cases}
		0, &\text{if } k\in[i,j];\\ d(k,i), &\text{if } i\in[k,j];\\ d(k,j),& \text{if } j\in [k,i].
	\end{cases}
\end{equation}
Also,
\begin{equation}\label{eq:cotd}
	d_{i,j}^{k}\leq\operatorname{min}\{d(k,i),d(k,j)\} \qquad\text{and}\qquad d_{i,j}^k+d_{k,j}^i=d(k,i).
\end{equation}

\subsection{Simple Prime Modules and $q$-Factors}\label{ss:psm}
Although this subsection is contained in \cite[Section 3.3]{mosi}, we chose to partially reproduce it here for easy of referencing later on.

Given $(i,r),(j,s)\in I\times  \mathbb Z_{>0}$, there exists a finite set $\mathscr R_{i,j}^{r,s} \subseteq \mathbb Z_{>0}$
such that
\begin{equation}\label{defredset}
	L_q(\bs\omega_{i,a,r})\otimes L_q(\bs\omega_{j,b,s}) \text{ is reducible}\qquad\Leftrightarrow\qquad \frac{a}{b} = q^m \text{ with } |m|\in \mathscr R_{i,j}^{r,s}.
\end{equation}
Moreover, in that case,
\begin{equation}\label{e:krhwtp}
	L_q(\bs\omega_{i,a,r})\otimes L_q(\bs\omega_{j,b,s}) \text{ is  highest-$\ell$-weight}\qquad\Leftrightarrow\qquad m>0.
\end{equation}
If $\lie g$ is of type $A$  and $i,j\in I, r,s\in\mathbb Z_{>0}$, we have 
\begin{equation}\label{t:krredsets}
		\mathscr R_{i,j}^{r,s} = \{r+s+d(i,j)-2p: - d([i,j],\partial I)\le p<\min\{r,s\}  \}. 
\end{equation}
Given a connected subdiagram $J$ such that $[i,j]\subseteq J$, let $\mathscr R_{i,j,J}^{r,s}$ be determined by 
\begin{equation*}
	V_q((\bs\omega_{i,a,r})_J)\otimes V_q((\bs\omega_{j,b,s})_J)  \text{ is reducible}\qquad\Leftrightarrow\qquad \frac{a}{b} = q^m \text{ with } |m|\in \mathscr R_{i,j,J}^{r,s}.
\end{equation*}
\Cref{c:sJs} below implies
\begin{equation}\label{incredsets}
	\mathscr R_{i,j,J}^{r,s}\subseteq \mathscr R_{i,j,K}^{r,s} \quad\text{if}\quad J\subseteq K.
\end{equation}
Also, setting
\begin{equation}\label{e:redsetsl2}
	\mathscr R_{i}^{r,s} =  \mathscr R_{i,i,\{i\}}^{r,s},
\end{equation}
we have
\begin{equation}\label{c:krtpsl2}
	\mathscr R_{i}^{r,s}=\{d_i(r+s-2p):0\le p<\min\{r,s\}\} \quad\text{for every}\quad i\in I,\ r,s\in\mathbb Z_{>0}.
\end{equation}

\begin{prop}[{\cite[Proposition 3.3.4]{mosi}}]\label{p:qftp}
	If $\bs\pi,\bs\varpi\in\mathcal P^+$ are such that $L_q(\bs\pi)\otimes L_q(\bs\varpi)$ is simple, then they have dissociate $q$-factorizations.\qed
\end{prop}

\subsection{Factorization Graphs}\label{ss:prefact}
We again ask the reader to refer to Sections 2.1 and 2.2 of \cite{mosi} for the basic notation regarding graph theory in order to avoid excessive repetition of background. 
For easy of referencing, we recall the main parts of \cite[Sections 2.5 and 3.4]{mosi}, where the notion of $q$-factorization graphs was introduced.

A pre-factorization graph is a directed graph $G$ equipped with three maps
\begin{equation*}
	c:\mathcal V\to I, \quad \lambda:\mathcal V\to \mathbb Z_{>0}, \quad \epsilon:\mathcal A\to \mathbb Z_{>0},
\end{equation*}
called the coloring, the weight, and the exponent, respectively,  and $\epsilon$ satisfies the following compatibility condition:
\begin{equation}\label{e:expcomp}
	\epsilon_\rho = \epsilon_{\rho'} \qquad\text{for all}\qquad \rho,\rho'\in\mathscr P_{v,v'}, v,v'\in \mathcal V.
\end{equation}
Here, $\mathscr P_{v,v'}$ is the set of paths from $v$ to $v'$ and 
\begin{equation*}
	\epsilon_\rho = \sum_{j=1}^m s_j\epsilon(a_j),
\end{equation*}
where $\rho=e_1\cdots e_m$ has signature  $\sigma_\rho=(s_1,\dots,s_m)$ and $e_j$ is the edge associated to the arrow $a_j$.
Set also
\begin{equation}\label{e:incpaths}
	\mathscr P^+_G = \{\rho\in\mathscr P_G:\epsilon_\rho>0\} \quad\text{and}\quad \mathscr P^-_G = \{\rho\in\mathscr P_G:\epsilon_\rho<0\}.
\end{equation}
The structures maps of  a pre-factorization graph can be locally represented by a picture of the form
	\begin{tikzcd}
		\stackrel{r}{i} \arrow[r,"m"] & \stackrel{s}{j}
		\end{tikzcd} 
where $i$ and $j$ are the colors at the corresponding vertices, $r$ and $s$ are their associated weights and $m$ is the exponent associated to the given arrow. Condition \eqref{e:expcomp} implies a pre-factorization graph contains no oriented cycles.
In particular, the set $\mathcal A$ induces a partial order on $\mathcal V$  by the transitive extension of the strict relation
\begin{equation*}
	h_a\prec t_a \quad\text{for}\quad a\in\mathcal A.
\end{equation*}

Recall  \eqref{defredset} and \eqref{e:redsetsl2}. A pre-factorization graph $G$ is said to be a $q$-factorization graph if,
for every $i\in I$, 
\begin{equation}\label{e:areqfact}
	v,v'\in \mathcal V_i, \ \rho\in\mathscr P_{v,v'}  \qquad\Rightarrow\qquad |\epsilon_\rho| \notin \mathscr R_i^{\lambda(v),\lambda(v')}
\end{equation}
and
\begin{equation}\label{e:notherarrow}
	\rho\in\mathscr P_{v,v'}\cap\mathscr P_G^+  \quad\text{with}\quad \epsilon_\rho \in \mathscr R_{c(v),c(v')}^{\lambda(v),\lambda(v')} \qquad\Rightarrow\qquad (v',v)\in\mathcal A.
\end{equation}
We  refer to a pre-factorization graph satisfying \eqref{e:notherarrow} as a pseudo $q$-factorization graph.

If $G$ is a connected pre-factorization graph, for each choice of $(v_0,a)\in \mathcal V\times\mathbb F^\times$, we can associate a Drinfeld polynomial by 
\begin{equation}\label{e:polytograph}
	\bs\pi_{G,v_0,a}=\prod_{v\in \mathcal V} \bs\omega_{c(v),a_v,\lambda(v)},
\end{equation}
where $a_{v_0}=a$ and $a_{v}=aq^{\epsilon_\rho}$ if $\rho\in\mathscr P_{v_0,v}$.
Condition \eqref{e:expcomp} guarantees this is well-defined. Conversely,  any pseudo $q$-factorization  of a Drinfeld polynomial $\bs\pi$ gives rise  to a pseudo $q$-factorization graph which is a $q$-factorization graph if and only if it is the $q$-factorization of $\bs\pi$. The vertex set $\mathcal V$ is the multiset of (pseudo) $q$-factors, the coloring is determined by 
\begin{equation*}
	c^{-1}(\{i\}) := \mathcal V_i := \{\bs\omega\in \mathcal V: \supp(\bs\omega)=\{i\}\}, \quad i\in I,
\end{equation*} 
and the weight map $\lambda:\mathcal V\to\mathbb Z_{>0}$ is defined by
\begin{equation}
	\lambda(\bs\omega) = \wt(\bs\omega)(h_i) \quad\text{for all}\quad \bs\omega\in \mathcal V_i.
\end{equation}
In particular,
\begin{equation}
	\sum_{i\in I}\sum_{\bs\omega\in \mathcal V_i} \lambda(\bs\omega)\omega_i = \wt(\bs\pi).
\end{equation}
The set of arrows $\mathcal A=\mathcal A(\bs\pi)$ is defined as the set of ordered pairs of 
$q$-factors, say $(\bs\omega_{i,a,r},\bs\omega_{j,b,s})$, such that
\begin{equation}\label{e:arrowsdefp}
	a=bq^m \quad\text{for some}\quad m\in \mathscr R_{i,j}^{r,s}.
\end{equation} 
This is equivalent to saying that $L_q(\bs\omega_{i,a,r})\otimes L_q(\bs\omega_{j,b,s})$
is reducible and highest-$\ell$-weight. In the case of the actual $q$-factorization, we necessarily have $m\notin \mathscr R_i^{r,s}$ when $i=j$. The value of the exponent $\epsilon:\mathcal A\to\mathbb Z_{>0}$ at an arrow satisfying \eqref{e:arrowsdefp} is set to be $m$.  We refer to $G$ as a pseudo $q$-factorization graph over $\bs\pi$. In the case this construction was performed using the $q$-factorization of $\bs\pi$, then $G$ is called the  $q$-factorization graph of $\bs\pi$ and it is denoted by $G(\bs\pi)$.

\begin{prop}[{\cite[Proposition 3.4.1]{mosi}}]\label{p:primeconect}
	Let $\bs\pi\in\mathcal P^+$. If $G_1, \cdots, G_k$ are the connected components of $G(\bs\pi)$ and $\bs\pi^{(j)}\in\mathcal P^+, 1\le j\le k$, are such that $\bs\pi = \prod_{j=1}^k \bs\pi^{(j)}$ and $G_j = G(\bs\pi^{(j)})$, then
	\begin{equation*}
		L_q(\bs\pi)\cong L_q(\bs\pi^{(1)})\otimes\cdots\otimes L_q(\bs\pi^{(k)}).
	\end{equation*}  
	In particular,  $G(\bs\pi)$ is connected if $L_q(\bs\pi)$ is prime.	
	\qed
\end{prop}

The converse of the above proposition is not true. In fact, \Cref{t:3lineprime} below implies that, if $\lie g$ is of type $A_2$ and $r=2$ in \eqref{e:newprimex}, the given graph  is not prime. On the other hand, the same theorem implies this graph is prime if $r\ge 3$.
	\begin{equation}\label{e:newprimex}
		\begin{tikzcd}
			\stackrel{r}{1} \arrow[r,"r+1"] & \stackrel{2}{2} & \stackrel{1}{1} \arrow[swap,l,"4"]
		\end{tikzcd} 
	\end{equation}
	Note that, for $r=1$, this is only a pre-factorization graph, but the corresponding simple module is prime and its $q$-factorization graph is:
	\begin{equation*}
		\begin{tikzcd}
			\stackrel{2}{2}  & \stackrel{2}{1} \arrow[swap,l,"3"]
		\end{tikzcd} 
	\end{equation*}
	The  simple modules associated to these graphs belong to the Hernandez-Leclerc category $\mathcal C_{r+1}$.

	Recall also the following notions of duality for pre-factorization graphs. Given a graph $G$, we denote by $G^-$ the graph obtained from $G$ by reversing all the arrows and keeping the rest of structure of (pre)-factorization graph. Then, $G^-$ is a factorization graph as well, which we refer to as the arrow-dual of $G$. Similarly, the graph $G^*$, called the color-dual of $G$, obtained  by changing the coloring according to the rule $i\mapsto i^*$ for all $i\in I$, is a factorization graph. Moreover,
\begin{equation}\label{e:dualgraphs}
	\bs\pi_{G^-,v,a^{-1}} =\bs\pi_{G,v,a}^- \qquad\text{and}\qquad \bs\pi_{G^*,v,aq^{-r^\vee h^\vee}} =\bs\pi_{G,v,a}^*.
\end{equation}

\subsection{Statement of the Main Results}\label{ss:main}
Throughout this section, we let $G=G(\bs\pi)=(\mathcal V,\mathcal A)$ be a $q$-factorization graph. We say $G$ is prime if $L_q(\bs\pi)$ is prime. We start by recalling:

\begin{prop}[{\cite[Proposition 3.5.2]{mosi}}]\label{p:buildupprimesmono}
	Suppose $G$ is prime and $\#\mathcal V>1$. Then, for every $v\in \partial G$,  $G_{\mathcal V\setminus\{v\}}$ is also prime.\qed
\end{prop}

 Together with elementary combinatorial properties of trees, \Cref{p:buildupprimesmono} implies:

\begin{cor}\label{c:subpt=>p}
	If $G$ is a prime tree, every of its proper connected subgraphs are  prime.\hfill\qed
\end{cor}

This corollary is false for general $q$-factorization graphs as  we show in \Cref{ex:cesubpt=>p}. Its converse is false even for trees as  \Cref{ex:cosubpt=>p} shows. The following proposition, which can also be regarded as a stronger version of this corollary, provides a criterion for checking if a tree is prime by studying its $3$-vertex subgraphs. It will be proved in \Cref{ss:redc}.

\begin{prop}\label{p:critsimp}
	Suppose $\bs\pi,\bs\pi'\in\mathcal P^+$ have dissociate $q$-factorizations and that $G(\bs\pi\bs\pi')$ is a tree. Suppose further that  the unique arrow 
	connecting $G(\bs\pi)$ and $G(\bs\pi')$ in $G(\bs\pi\bs\pi')$ is of the form $(\bs\omega,\bs\omega')$ with $\bs\omega\in G(\bs\pi), \bs\omega'\in G(\bs\pi')$, and that there exists $\widetilde{\bs\omega}\in\mathcal P^+$ such that one of the following conditions holds:
	\begin{enumerate}[(i)]
		\item $(\bs\omega,\widetilde{\bs\omega})\in\mathcal A_G$ and $L_q(\widetilde{\bs{\omega}}\bs{\omega})\otimes L_q(\bs{\omega}')$ is simple;
		\item $(\widetilde{\bs\omega},\bs\omega')\in\mathcal A_{G'}$ and $L_q(\widetilde{\bs{\omega}}\bs\omega')\otimes L_q(\bs{\omega})$ is simple.
	\end{enumerate}
	Then, $L_q(\bs\pi)\otimes L_q(\bs\pi')$ is simple, as well.
\end{prop}

The following criterion for determining whether a tree is prime  will follow as  consequence of \Cref{p:critred2}. 

\begin{prop}\label{c:nonadcjprim}
	If  $G$ is a tree and $L_q(\bs{\omega})^*\otimes L_q(\bs{\omega}^{\prime})$ is simple for any pair of non-adjacent vertices $\bs{\omega}$ and $\bs{\omega}^{\prime}$ of $G$, then $L_q(\bs{\pi})$ is prime.\endd
\end{prop}

Let us recall the main result of \cite{mosi}.

\begin{thm}[{\cite[Theorem 3.5.5]{mosi}}]\label{t:toto}
	If $\lie g$ is of type $A$, every totally ordered $q$-factorization graph is prime. \qed
\end{thm}

As we have already pointed out in \eqref{e:newprimex}, even for type $A_2$, it is not true that every tree with three vertices is prime. There are just three possibilities of connected graphs with three vertices: alternating lines, monotonic lines,  and a triangles. The latter two are prime by \Cref{t:toto}.  If $G$ is an alternating line, say
	\begin{equation*}
	\begin{tikzcd}
		\stackrel{r_1}{i_1} & \arrow[swap,l,"m_1"]  \stackrel{r}{i} \arrow[r,"m_2"] & \stackrel{r_2}{i_2} 
	\end{tikzcd} \qquad\text{or}\qquad 
	\begin{tikzcd}
		\stackrel{r_1}{i_1} \arrow[r,"m_1"] &   \stackrel{r}{i}  & \arrow[swap,l,"m_2"] \stackrel{r_2}{i_2} 
	\end{tikzcd}
\end{equation*}
we have the following characterization for type $A$ proved in \Cref{ss:3linecut}. 

\begin{thm}\label{t:3lineprime}
	Assume $\lie g$ is of type $A$ and let $G$ be an alternating line as above. For $j=1,2$, let also $I_j\subseteq I$ be the minimal connected subdiagram containing $[i,i_j]$ such that $m_j\in\mathscr R_{i,i_j,I_j}^{r,r_j}$ and let $j'$ be such that $\{j,j'\}=\{1,2\}$. Then, $G$ is not prime if and only if there exists $j\in\{1,2\}$ such that
	\begin{equation*}\label{2gencondm}
		i_{j'}\in I_j, \qquad m_{j'}\in{\mathscr{R}_{i,i_{j'},I_j}^{r,r_{j'}}}, \qquad |m_j-m_{j'}-\check h_{I_j}|\in\mathscr R_{w_0^{I_j}(i_j),i_{j'},I_j}^{r_j,r_{j'}},
	\end{equation*}
	and
	\begin{equation*}\label{2exracondm}
		m_j-m_{j'}+1\notin\mathscr{R}_{i_j,i_{j'},I_j}^{r_j-1,r_{j'}}  \quad\text{if}\quad r_j>1.
	\end{equation*}\endd
\end{thm}

The present proof of \Cref{t:3lineprime} utilizes the precise description of the sets $\mathscr R_{i,j}^{r,s}$ and, hence, in order to extend it to other types, it requires a case by case analysis. The proof consists of showing that the cut obtained by isolating the $i_j$-colored vertex, $j=1,2$, is simple if and only if the stated conditions are satisfied.

An important concept related to the connection of finite-dimensional representations of quantum affine algebras with the theory of cluster algebras is that of a real simple module, that is a simple module whose tensor square is simple.  We shall say $G$ is real if $L_q(\bs\pi)$ is real. 

\begin{con}\label{cj:realtree}
	Every tree is real.\endd
\end{con}

The last of our main results is: 

\begin{thm}\label{t:realtree}
	\Cref{cj:realtree} is true if $\lie g$ is of type $A$.\endd
\end{thm}

The only reason the proof of this theorem requires the assumption that $\lie g$ is of type $A$ is that it uses the following corollary of \Cref{t:3lineprime}, proved in \Cref{ss:realtree}.  Hence,  \Cref{cj:realtree} is proved if this corollary holds for all types, which is expected to be true. 

\begin{cor}\label{c:3aline}
	If $\lie g$ is of type $A$,  $i,j\in I, a\in\mathbb F^\times, r,s\in\mathbb Z_{>0}$, and $m\in\mathscr R_{i,j}^{r,s}$, the module
	$L_q(\bs\omega_{i,a,r})\otimes L_q(\bs\omega_{i,a,r}\bs\omega_{j,aq^m,s})$ is simple.\endd
\end{cor}

\subsection{Examples}\label{ss:ex}

\begin{ex}\label{ex:cosubpt=>p}
	We  give a counterexample to the converse of  \Cref{c:subpt=>p}. More precisely, we give an example of a non-prime tree all of whose proper connected subgraphs are prime. Let  $\mathfrak{g}$ be of type $A_n, n\ge 3,$ and consider $\bs{\pi}=\bs{\omega}_{1,aq,2}\bs{\omega}_{2,aq^5}\bs{\omega}_{3,aq^6,3}\bs{\omega}_{3,aq^8}$, where $a\in\mathbb{F}^{\times}$, whose $q$-factorization graph is
		\begin{equation*}
		\begin{tikzcd}
			\stackrel{1}{3} \arrow[r,"3"] &   \stackrel{1}{2} \arrow[r,"4"] & \stackrel{2}{1} & \arrow[swap,l,"5"] \stackrel{3}{3}
		\end{tikzcd}
	\end{equation*}
	Theorems \ref{t:toto} and \ref{t:3lineprime} imply the two connected subgraphs with three vertices are prime. We can be more precise regarding the  subgraph which is an alternating line using the results of \Cref{ss:3linecut}: the tensor products
	\begin{equation*}
		L_q(\bs{\omega}_{3,aq^6,3}\bs{\omega}_{1,aq,2})\otimes L_q(\bs{\omega}_{2,aq^5})\quad\textrm{and}\quad L_q(\bs{\omega}_{2,aq^5}\bs{\omega}_{1,aq,2})\otimes L_q(\bs{\omega}_{3,aq^6,3})
	\end{equation*}
	are both reducible. Indeed, the first one does not satisfy the first condition in \eqref{2gencondm} while the second (which is not highest-$\ell$-weight by the more precise analysis of \Cref{ss:3linecut}), does not satisfy the third condition.  However, we will see that $G(\bs\pi)$ is not prime by checking that 
	\begin{equation*}
		V=L_q(\bs{\omega}_{3,aq^8}\bs{\omega}_{2,aq^5}\bs{\omega}_{1,aq,2})\otimes L_q(\bs{\omega}_{3,aq^6,3}) \quad\text{is simple.}
	\end{equation*}

	The argument that follows utilizes results reviewed in \Cref{ss:hlwtp}.	
	In order to show $V$ is simple, since the tensor product in the opposite order is highest-$\ell$-weight by \eqref{e:krhwtp} and \Cref{l:hlwquot}, \Cref{c:vnvstar} implies it suffices to show $V$ is highest-$\ell$-weight as well.  To do that, consider 
	\begin{equation*}
		W=L_q(\bs{\omega}_{3,aq^8}\bs{\omega}_{2,aq^5}\bs{\omega}_{1,aq,2})\otimes L_q(\bs{\omega}_{3,aq^8})\otimes L_q(\bs{\omega}_{3,aq^5,2}).
	\end{equation*}
	Let
	\begin{align*}
		W_{1,2}= & L_q(\bs{\omega}_{3,aq^8}\bs{\omega}_{2,aq^5}\bs{\omega}_{1,aq,2})\otimes L_q(\bs{\omega}_{3,aq^8}),\\
		W_{1,3}= &L_q(\bs{\omega}_{3,aq^8}\bs{\omega}_{2,aq^5}\bs{\omega}_{1,aq,2})\otimes L_q(\bs{\omega}_{3,aq^5,2}),\\
		W_{2,3}= &L_q(\bs{\omega}_{3,aq^8})\otimes L_q(\bs{\omega}_{3,aq^5,2}).
	\end{align*}
	Notice $W_{1,3}$ is highest-$\ell$-weight. Therefore, we have epimorphisms
	\begin{equation*}
		W_{1,3}\twoheadrightarrow L_q(\bs{\omega}_{3,aq^8}\bs{\omega}_{3,aq^5,2})=L_q(\bs{\omega}_{3,aq^6,3})
	\end{equation*}
	and
	\begin{equation*}
		W=L_q(\bs{\omega}_{3,aq^8}\bs{\omega}_{2,aq^5}\bs{\omega}_{1,aq,2})\otimes W_{1,3}\twoheadrightarrow L_q(\bs{\omega}_{3,aq^8}\bs{\omega}_{2,aq^5}\bs{\omega}_{1,aq,2})\otimes L_q(\bs{\omega}_{3,aq^6,3})=V.
	\end{equation*}
	So, it suffices check that $W$ is highest-$\ell$-weight. In its turn, by Theorem \ref{cyc}, it is enough to show that $W_{1,2}$ and $W_{1,3}$ are highest-$\ell$-weight. 
	
	To see that $W_{1,2}$ is highest-$\ell$-weight, consider 
	\begin{equation*}
		\widetilde{W}_{1,2}=L_q(\bs{\omega}_{3,aq^8}\bs{\omega}_{2,aq^5})\otimes L_q(\bs{\omega}_{1,aq,2})\otimes L_q(\bs{\omega}_{3,aq^8})
	\end{equation*}
	and notice $L_q(\bs{\omega}_{3,aq^8}\bs{\omega}_{2,aq^5})\otimes L_q(\bs{\omega}_{1,aq,2})$ is highest-$\ell$-weight by Lemma \ref{l:hlwquot}. Therefore, there exists an epimorphism $\widetilde{W}_{1,2}\twoheadrightarrow W_{1,2}$ and we are left to show that $\widetilde{W}_{1,2}$ is highest-$\ell$-weight. Notice that $7=8-1\notin\mathscr{R}^{2,1}_{1,3}$. Thus, $L_q(\bs{\omega}_{1,aq,2})\otimes L_q(\bs{\omega}_{3,aq^8})$ is simple (and highest-$\ell$-weight). In its turn, $L_q(\bs{\omega}_{3,aq^8}\bs{\omega}_{2,aq^5})\otimes L_q(\bs{\omega}_{3,aq^8})$ is simple (and highest-$\ell$-weight) by \Cref{t:3lineprime}. \Cref{cyc} then implies $\widetilde{W}_{1,2}$ is highest-$\ell$-weight as desired. 
	
	To prove that $W_{1,3}$ is highest-$\ell$-weight, consider 
	\begin{equation*}
		\widetilde{W}_{1,3}=L_q(\bs{\omega}_{3,aq^8})\otimes L_q(\bs{\omega}_{2,aq^5}\bs{\omega}_{1,aq,2})\otimes L_q(\bs{\omega}_{3,aq^5,2})
	\end{equation*}
	and notice $L_q(\bs{\omega}_{3,aq^8})\otimes L_q(\bs{\omega}_{2,aq^5}\bs{\omega}_{1,aq,2})$ is highest-$\ell$-weight by Lemma \ref{l:hlwquot}. Therefore, there exists an epimorphism $\widetilde{W}_{1,3}\twoheadrightarrow W_{1,3}$ and it suffices to show that $\widetilde{W}_{1,3}$ is highest-$\ell$-weight. Since  $L_q(\bs{\omega}_{3,aq^8})\otimes L_q(\bs{\omega}_{3,aq^5,2})$ is clearly highest-$\ell$-weight and $L_q(\bs{\omega}_{2,aq^5}\bs{\omega}_{1,aq,2})\otimes L_q(\bs{\omega}_{3,aq^5,2})$ is simple by  \Cref{t:3lineprime},  the claim follows from \Cref{cyc}. \endd
\end{ex}

\begin{ex}\label{ex:cesubpt=>p}
	We now give an example showing the hypothesis that $G$ is a tree in \Cref{c:subpt=>p} is essential. Let $\mathfrak{g}$ be of type $A_2$, $a\in\mathbb F^\times$, and consider $\bs{\pi}=\bs{\omega}_{1,aq^7,2}\bs{\omega}_{1,a}\bs{\omega}_{2,aq^4,2}\bs{\omega}_{2,aq^3}$, whose $q$-factorization graph is
	\begin{equation*}
		\begin{tikzcd}
			\stackrel{2}{2} \arrow[dr,"4"] & \arrow[swap,l,"3"]    \stackrel{2}{1} \arrow[r,"4"] & \arrow[swap,dl,"3"] \stackrel{1}{2}\\
			 &  \stackrel{1}{1}
		\end{tikzcd}
	\end{equation*}
	By \Cref{t:3lineprime}, 
	\begin{equation}\label{e:cesubpt=>pr}
		L_q(\bs{\omega}_{2,aq^3}\bs{\omega}_{1,a})\otimes L_q(\bs{\omega}_{2,aq^4,2})\quad\textrm{and}\quad L_q(\bs{\omega}_{2,aq^3})\otimes L_q(\bs{\omega}_{1,aq^7,2}\bs{\omega}_{2,aq^4,2}) \quad\text{are reducible,}
	\end{equation}
	while 
	\begin{equation}\label{e:cesubpt=>ps}
		L_q(\bs{\omega}_{2,aq^4,2}\bs{\omega}_{1,a})\otimes L_q(\bs{\omega}_{2,aq^3})\quad\textrm{and}\quad L_q(\bs{\omega}_{2,aq^4,2})\otimes L_q(\bs{\omega}_{1,aq^7,2}\bs{\omega}_{2,aq^3}) \quad\text{are simple,}
	\end{equation}
	showing that not every proper connected subgraph is prime. Let us check that $L_q(\bs{\pi})$ is prime. This is equivalent to showing that the following tensor products are reducible:
	\begin{enumerate}[1)]
		\item $L_q(\bs{\omega}_{1,aq^7,2})\otimes L_q(\bs{\omega}_{2,aq^4,2}\bs{\omega}_{2,aq^3}\bs{\omega}_{1,a})$;
		\item $L_q(\bs{\omega}_{1,aq^7,2}\bs{\omega}_{2,aq^3}\bs{\omega}_{1,a})\otimes L_q(\bs{\omega}_{2,aq^4,2})$;
		\item $L_q(\bs{\omega}_{2,aq^3})\otimes L_q(\bs{\omega}_{1,aq^7,2}\bs{\omega}_{2,aq^4,2}\bs{\omega}_{1,a})$;
		\item $L_q(\bs{\omega}_{1,a})\otimes L_q(\bs{\omega}_{1,aq^7,2}\bs{\omega}_{2,aq^4,2}\bs{\omega}_{2,aq^3})$;
		\item $L_q(\bs{\omega}_{1,aq^7,2}\bs{\omega}_{1,a})\otimes L_q(\bs{\omega}_{2,aq^4,2}\bs{\omega}_{2,aq^3})$;
		\item $L_q(\bs{\omega}_{2,aq^3}\bs{\omega}_{1,a})\otimes L_q(\bs{\omega}_{1,aq^7,2}\bs{\omega}_{2,aq^4,2})$;
		\item $L_q(\bs{\omega}_{2,aq^4,2}\bs{\omega}_{1,a})\otimes L_q(\bs{\omega}_{1,aq^7,2}\bs{\omega}_{2,aq^3})$.
	\end{enumerate}
	
	\noindent 1) It follows from \eqref{e:cesubpt=>ps} that we have an isomorphism 
	\begin{equation*}
		L_q(\bs{\omega}_{1,aq^7,2})\otimes L_q(\bs{\omega}_{2,aq^4,2}\bs{\omega}_{2,aq^3}\bs{\omega}_{1,a})\cong L_q(\bs{\omega}_{1,aq^7,2})\otimes L_q(\bs{\omega}_{2,aq^3})\otimes L_q(\bs{\omega}_{2,aq^4,2}\bs{\omega}_{1,a}),
	\end{equation*}
	which is reducible because $L_q(\bs{\omega}_{1,aq^7,2})\otimes L_q(\bs{\omega}_{2,aq^3})$ is reducible since $4=7-3\in\mathscr R_{1,2}^{2,1}$. 
	
	\noindent 2) We claim the given tensor product is not highest-$\ell$-weight and, hence, reducible. Indeed, if this were not the case, \cite[Proposition 5.2.2]{mosi} would imply that $L_q(\bs{\omega}_{2,aq^3}\bs{\omega}_{1,a})\otimes L_q(\bs{\omega}_{2,aq^4,2})$ is highest-$\ell$-weight, as well. Since $4>\min\{3,0\}$, $L_q(\bs{\omega}_{2,aq^4,2})\otimes L_q(\bs{\omega}_{2,aq^3}\bs{\omega}_{1,a})$ is  highest-$\ell$-weight by \eqref{e:krhwtp} and \Cref{l:hlwquot}. \Cref{c:vnvstar} would then imply  $L_q(\bs{\omega}_{2,aq^3}\bs{\omega}_{1,a})\otimes L_q(\bs{\omega}_{2,aq^4,2})$ is simple, contradicting \eqref{e:cesubpt=>pr}. 
	
	\noindent 3) Similarly to the previous case, $L_q(\bs{\omega}_{2,aq^3})\otimes L_q(\bs{\omega}_{1,aq^7,2}\bs{\omega}_{2,aq^4,2}\bs{\omega}_{1,a})$ is not highest-$\ell$-weight since, otherwise, \cite[Proposition 5.2.2]{mosi} would imply $L_q(\bs{\omega}_{2,aq^3})\otimes L_q(\bs{\omega}_{1,aq^7,2}\bs{\omega}_{2,aq^4,2})$ were highest-$\ell$-weight, and then simple, contradicting \eqref{e:cesubpt=>pr}.
	
	\noindent 4) Similarly to case (1), we have an isomorphism
	\begin{equation*}
		L_q(\bs{\omega}_{1,a})\otimes L_q(\bs{\omega}_{1,aq^7,2}\bs{\omega}_{2,aq^4,2}\bs{\omega}_{2,aq^3})\cong L_q(\bs{\omega}_{1,a})\otimes L_q(\bs{\omega}_{2,aq^4,2})\otimes L_q(\bs{\omega}_{1,aq^7,2}\bs{\omega}_{2,aq^3}),
	\end{equation*}
	and we are done since  $L_q(\bs{\omega}_{1,a})\otimes L_q(\bs{\omega}_{2,aq^4,2})$ is reducible. 
	
	\noindent 5) In this case we have 
	\begin{equation*}
		L_q(\bs{\omega}_{1,aq^7,2}\bs{\omega}_{1,a})\otimes L_q(\bs{\omega}_{2,aq^4,2}\bs{\omega}_{2,aq^3})\cong L_q(\bs{\omega}_{1,aq^7,2})\otimes L_q(\bs{\omega}_{1,a})\otimes L_q(\bs{\omega}_{2,aq^4,2})\otimes L_q(\bs{\omega}_{2,aq^3}),
	\end{equation*}
	and $L_q(\bs{\omega}_{1,a})\otimes L_q(\bs{\omega}_{2,aq^4,2})$ is reducible.
	
	\noindent 6) If $L_q(\bs{\omega}_{2,aq^3}\bs{\omega}_{1,a})\otimes L_q(\bs{\omega}_{1,aq^7,2}\bs{\omega}_{2,aq^4,2})$ were highest-$\ell$-weight, then \cite[Corollary 4.3.2]{mosi}  would imply that $L_q(\bs{\omega}_{2,aq^3}\bs{\omega}_{1,a})\otimes L_q(\bs{\omega}_{2,aq^4,2})$ is highest-$\ell$-weight as well and, hence, simple, yielding a contradiction as before. 
	
	\noindent 7) To see that $L_q(\bs{\omega}_{2,aq^4,2}\bs{\omega}_{1,a})\otimes L_q(\bs{\omega}_{1,aq^7,2}\bs{\omega}_{2,aq^3})$ is reducible, notice that $(\bs{\omega}_{1,aq^7,2}\bs{\omega}_{2,aq^3})^*=\bs{\omega}_{2,aq^4,2}\bs{\omega}_{1,a}$. Therefore, we have an epimorphism
	\begin{equation*}
		L_q(\bs{\omega}_{2,aq^4,2}\bs{\omega}_{1,a})\otimes L_q(\bs{\omega}_{1,aq^7,2}\bs{\omega}_{2,aq^3})\twoheadrightarrow\mathbb{F}
	\end{equation*}
	and, hence,  $L_q(\bs{\omega}_{2,aq^4,2}\bs{\omega}_{1,a})\otimes L_q(\bs{\omega}_{1,aq^7,2}\bs{\omega}_{2,aq^3})$ is reducible. 	\endd
\end{ex}	

Let us also mention a few examples related to the problem of determining the reality of certain modules.

\begin{ex}
	The converse of \Cref{cj:realtree} is not true, i.e., there are real modules whose $q$-factorizations graphs are not trees. Indeed, if $G=G(\bs\pi)$ is real, then $G'=G(\bs\pi^2)$ is real but, it is neither prime nor a tree. For instance, if $G$ is of the form 	\begin{tikzcd}
		\stackrel{r}{i} \arrow[r,"m"] & \stackrel{s}{j}
	\end{tikzcd} 	then $G'$ is of the form
	\begin{equation*}
		\begin{tikzcd}
			\stackrel{r}{i} \arrow[r,"m"] \arrow[dr,"m"]  & \stackrel{s}{j} & 	\stackrel{r}{i} \arrow[swap,l,"m"] \arrow[swap,dl,"m"] \\
			& \stackrel{s}{j} & 
		\end{tikzcd}
	\end{equation*}
	\endd
\end{ex}

\begin{ex}
	In contrast to the real graph $G'$ in the previous example, it was shown in \cite[Section 13.6]{hele:cluster} that the following graph for type $A_4$ is not real:
	\begin{equation*}
		\begin{tikzcd}
			\stackrel{1}{1}  \arrow[dr,"3"]  & \arrow[swap,l,"3"] \stackrel{1}{2} \arrow[r,"3"] & 	\stackrel{1}{3}  \arrow[swap,dl,"3"] \\
			& \stackrel{1}{2} & 
		\end{tikzcd}
	\end{equation*}
 In fact, this graph makes sense for type $A_n,n\ge 3$ and is not real (cf. \cite[Example 8.7]{cdfl}). In \cite[Example 8.6]{cdfl} it was shown the following graph is not real for type $A_n, n\ge 2$:
 \begin{equation*}
 	\begin{tikzcd}
 		\stackrel{2}{1}  \arrow[dr,"4"]  & \arrow[swap,l,"4"] \stackrel{3}{2} \arrow[r,"4"] & 	\stackrel{2}{1}  \arrow[swap,dl,"4"] \\
 		& \stackrel{1}{2} & 
 	\end{tikzcd}
 \end{equation*}
	\endd
\end{ex}

\section{Further Background and First Proofs}\label{s:crit}

\subsection{Hopf Algebra Facts} For easy of referencing, we partially reproduce \cite[Section 2.6]{mosi}. 

Given a Hopf algebra $\mathcal H$ over $\mathbb F$, its category $\mathcal C$ of finite-dimensional representations is an abelian monoidal category and we denote the (right) dual of a module  $V$ by $V^*$. More precisely, the action of $\mathcal H$ of $V^*$ is given by
\begin{equation}\label{rightdual}
	(hf)(v) = f(S(h)v) \quad\text{for}\quad h\in\mathcal H, f\in V^*, v\in V.
\end{equation}
It is well known that 
\begin{equation*}
	\operatorname{Hom}_{\mathcal H}(\mathbb F, V\otimes V^*) \ne 0 \qquad\text{and}\qquad \operatorname{Hom}_{\mathcal H}(V^*\otimes V, \mathbb F)\ne 0.
\end{equation*}
If the antipode is invertible, the notion of left dual module is obtained by replacing $S$ by $S^{-1}$ in \eqref{rightdual}. The left dual of $V$ will be denoted by $^*V$ and we have
\begin{equation*}
	^*(V^*)\cong (^*V)^* \cong V.
\end{equation*}
Given $\mathcal H$-modules $V_1,V_2,V_3$, we have
\begin{equation}\label{e:frobrec}
	\begin{aligned}
		\operatorname{Hom}_{\mathcal C}(V_1\otimes V_2, V_3)\cong \operatorname{Hom}_{\mathcal C}(V_1, V_3\otimes V_2^*), \quad
		\operatorname{Hom}_{\mathcal C}(V_1, V_2\otimes V_3)\cong \operatorname{Hom}_{\mathcal C}(V_2^*\otimes V_1, V_3),
	\end{aligned}
\end{equation}
and 
\begin{equation}\label{e:dualtp}
	(V_1\otimes V_2)^*\cong V_2^*\otimes V_1^*.
\end{equation}
Also, given a short exact sequence $0\to V_1\to V_2\to V_3\to 0$, we can consider the following short exact sequence
\begin{equation}\label{e:dualses}
	0\to V_3^*\to V_2^*\to V_1^*\to 0.
\end{equation}

\begin{lem}
	\label{c:nonzeromorph}
	Let $V_1,V_2,V_3, L_1,L_2$ be $\mathcal H$-modules and assume  $V_2$ is simple. If
	\begin{equation*}
		\varphi_1:L_1\rightarrow V_1\otimes V_2\quad\textrm{and}\quad\varphi_2: V_2\otimes V_3\rightarrow L_2
	\end{equation*}
	are nonzero homomorphisms, the composition
	$$L_1\otimes V_3 \xrightarrow{\varphi_1\otimes \operatorname{id}_{V_3}} V_1\otimes V_2\otimes V_3 \xrightarrow{\operatorname{id}_{V_1}\otimes \varphi_2} V_1\otimes L_2
	$$
	does not vanish. Similarly, if 
	\begin{equation*}
		\varphi_1:V_1\otimes V_2\rightarrow L_1\quad\textrm{and}\quad\varphi_2:L_2\rightarrow V_2\otimes V_3
	\end{equation*}
	are nonzero homomorphisms, the composition
	$$V_1\otimes L_2\xrightarrow{\operatorname{id}_{V_1}\otimes\varphi_2} V_1\otimes V_2\otimes V_3 \xrightarrow{\varphi_1\otimes\operatorname{id}_{V_3}} L_1\otimes V_3$$
	does not vanish.\qed
\end{lem}

\subsection{Tensor Products and Diagram Subalgebras}
The following result will be used in a crucial moment during the proof of \Cref{t:3lineprime}. A slightly modified argument which avoids the use of \Cref{rsh} is given in \cite{sil}.

\begin{thm}[\cite{kkko}]\label{rsh}
	Let $M$ and $N$ be finite-dimensional simple $U_q(\tilde{\mathfrak{g}})$-modules. Assume further that $M$ is real. Then $M\otimes N$ has a simple socle and a simple head. Moreover, if $\operatorname{soc}(M\otimes N)$ and $\operatorname{hd}(M\otimes N)$ are isomorphic, then $M\otimes N$ is simple.\hfill\qedsymbol
\end{thm}

If $V$ is a highest-$\ell$-weight module with highest-$\ell$-weight vector $v$ and $J\subset I$, we let $V_J$ denote the $U_q(\tlie g)_J$-submodule of $L_q(\bs{\pi})$ generated by $v$. Evidently, if $\bs\pi$ is the highest-$\ell$-weight of $V$, then $V_J$ is highest-$\ell$-weight with highest $\ell$-weight $\bs\pi_J$. Moreover, we have the following well-known facts:

\begin{equation}\label{e:weightsJ}
	V_J = \bigoplus_{\eta\in Q_J^+} V_{\wt(\bs\pi)-\eta}  = \bigoplus_{\eta\in Q_J} V_{\wt(\bs\pi)+\eta}.
\end{equation}

\begin{lem}\label{sl32sl2}
	If $V$ is simple, so is $V_J$.\hfill\qed
\end{lem}

Since $U_q(\tlie g)_J$ is not a sub-coalgebra of $U_q(\tlie g)$, if $M$ and $N$ are $U_q(\tlie g)_J$-submodules of $U_q(\tlie g)$-modules $V$ and $W$, respectively, it is in general not true that $M\otimes N$ is a $U_q(\tlie g)_J$-submodule of $V\otimes W$.  Recalling that we have an algebra isomorphism $U_q(\tlie g)_J\cong U_{q_J}(\tlie g_J)$, we shall denote by $M\otimes_J N$ the $U_q(\tlie g)_J$-module obtained by using the coalgebra structure from $ U_{q_J}(\tlie g_J)$. The next result describes a special situation on which $M\otimes N$ is a submodule isomorphic to $M\otimes_J N$.

\begin{prop}[{\cite[Proposition 2.2]{cp:minsl}}]\label{p:subJtp}
	Let $V$ and $W$ be finite-dimensional highest-$\ell$-weight modules with highest $\ell$-weights $\bs\pi,\bs\varpi\in\mathcal P^+$, respectively, and let $J\subseteq I$ be a connected subdiagram. Then, $V_J\otimes W_J$ is a $U_q(\tlie g)_J$-submodule of $V\otimes W$ isomorphic to $V_J\otimes_J W_J$ via the identity map.\hfill\qed
\end{prop}

\begin{cor}\label{c:sJs}
	In the notation of \Cref{p:subJtp}, if $V\otimes W$ is highest-$\ell$-weight, so is $V_J\otimes W_J$. Moreover, if $V\otimes W$ is simple, so is $V_J\otimes W_J$.\qed
\end{cor}

\begin{lem}[{cf. \cite[Lemma 5.4]{cp:dorey}}]\label{p:subJrest}
	Let $U,V,W$ be finite-dimensional highest-$\ell$-weight modules with highest $\ell$-weights $\bs\omega,\bs\pi,\bs\varpi\in\mathcal P^+$, respectively, and  let $J\subseteq I$ be a connected subdiagram. 
	If
	\begin{equation}\label{e:subJrest}
		\wt(\bs\pi)+ \wt(\bs\varpi)-\wt(\bs\omega) \in Q_J^+,
	\end{equation}
	restriction induces natural  linear maps  ${\rm Hom}_{U_q(\tlie g)}(V\otimes W,U) \to {\rm Hom}_{U_q(\tlie g)_J}(V_J\otimes W_J,U_J)$ and ${\rm Hom}_{U_q(\tlie g)}(U,V\otimes W) \to {\rm Hom}_{U_q(\tlie g)_J}(U_J,V_J\otimes W_J)$. Moreover, the latter is injective  and the former is injective if $U$ is simple.
\end{lem}

\begin{proof}
	It was shown in the proof of  \Cref{p:subJtp} that
	\begin{equation}\label{e:sJs}
		V_J\otimes W_J = \opl_{\eta\in Q_J^+}^{} (V\otimes W)_{\wt(\bs\pi)+\wt(\bs\varpi)-\eta}.
	\end{equation}	
	Let $f\in {\rm Hom}_{U_q(\tlie g)}(V\otimes W,U)$ and denote by $f_J$ its restriction to $V_J\otimes W_J$. It follows from \eqref{e:subJrest}, \eqref{e:weightsJ}, and \eqref{e:sJs} that the image of $f_J$ lies in $U_J$. Thus, we have a natural linear map  $\phi:{\rm Hom}_{U_q(\tlie g)}(V\otimes W,U) \to {\rm Hom}_{U_q(\tlie g)_J}(V_J\otimes W_J,U_J), f\mapsto f_J$. If $U$ is simple and $f\ne 0$, then the image of $f$ contains a highest-$\ell$-weight vector $u$ for $U$ which, by \eqref{e:weightsJ} and \eqref{e:sJs}, must be in the image of $f_J$. It follows that $f_J\ne 0$, showing $\phi$ is injective.
	
	Now let $f\in {\rm Hom}_{U_q(\tlie g)}(U,V\otimes W)$ and denote by $f_J$ its restriction to $U_J$. It follows from \eqref{e:subJrest}, \eqref{e:weightsJ}, and \eqref{e:sJs} that the image of $f_J$ lies in  $V_J\otimes W_J$ and, hence,  we have a natural linear map $\phi:{\rm Hom}_{U_q(\tlie g)}(U,V\otimes W) \to {\rm Hom}_{U_q(\tlie g)_J}(U_J,V_J\otimes W_J), f\mapsto f_J$. Let $u$ be a highest-$\ell$-weight vector for $U$. Then, $f$ is completely determined by its value on $u$. In particular, $f_J=0$ if and only if $f=0$ showing $\phi$ is injective. 	 
\end{proof}

\subsection{Highest-$\ell$-weight Tensor Products}\label{ss:hlwtp}
We record some results from \cite{mosi} which will be used in the proof of \Cref{t:3lineprime}.

\begin{lem}[{\cite[Lemma 4.1.2]{mosi}}]\label{l:subsandq}
	Let $\bs\pi,\bs\pi'\in\mathcal P^+$, $V=L_q(\bs\pi)\otimes L_q(\bs\pi')$, and $W=L_q(\bs\pi')\otimes L_q(\bs\pi)$.
	\begin{enumerate}[(a)]
		\item $V$ contains a submodule isomorphic to $L_q(\bs\varpi), \bs\varpi\in\mathcal P^+$, if and only there exists an epimorphism $W\to L_q(\bs\varpi)$.
		\item If $W$ is highest-$\ell$-weight, the submodule of $V$ generated by its top weight space is simple.
		\item If $V$ is not highest-$\ell$-weight, there exists an epimorphism $V\to L_q(\bs\varpi)$ for some $\bs\varpi\in\mathcal P^+$ such that $\bs\varpi<\bs\pi\bs\pi'$. \qed
	\end{enumerate}
\end{lem}

\begin{prop}[{\cite[Corollary 4.1.4]{mosi}}]\label{c:vnvstar}
	Let $\bs{\pi}, \bs\varpi\in\mathcal{P}^+$. Then, $L_q(\bs{\pi})\otimes L_q(\bs\varpi)$ is simple if and only if both $L_q(\bs{\pi})\otimes L_q(\bs\varpi)$ and $L_q(\bs\varpi)\otimes L_q(\bs{\pi})$ are highest-$\ell$-weight.\qed
\end{prop}

The ``if'' part  of the following theorem is the main result of \cite{Hernandez2} and a proof of the converse can be found in \cite[Theorem 4.1.5]{mosi}.

\begin{thm}\label{cyc}
	Let $S_1,\cdots, S_m$ be simple $U_q(\tlie g)$-modules. Then, $S_1\otimes\cdots\otimes S_m$ is highest-$\ell$-weight if and only if  $S_i\otimes S_j$ is highest-$\ell$-weight for all $1\leq i<j\leq m$. \qed
\end{thm}

The following corollary (\cite[Corollary 4.1.6]{mosi}) was prominently used in \cite{mosi} and will be even more frequently utilized here.
   
\begin{cor}\label{l:hlwquot}
	Given $\bs{\pi},\widetilde{\bs{\pi}}\in\mathcal{P}^+$, $L_q(\bs{\pi})\otimes L_q(\widetilde{\bs{\pi}})$ is highest-$\ell$-weight if there exist 
	$\bs\pi^{(k)}\in\mathcal P^+, 1\le k\le m$, $\widetilde{\bs\pi}^{(k)}\in\mathcal P^+, 1\le k\le \tilde m$, such that  
	\begin{equation*}
		\bs\pi = \prod_{k=1}^m \bs\pi^{(k)}, \quad  \widetilde{\bs\pi} = \prod_{k=1}^{\tilde m} \widetilde{\bs\pi}^{(k)}, 
	\end{equation*}
	and the following tensor products are highest-$\ell$-weight:
	\begin{equation*}
		L_q(\bs\pi^{(k)})\otimes L_q(\bs\pi^{(l)}), \quad L_q(\widetilde{\bs\pi}^{(k)})\otimes L_q(\widetilde{\bs\pi}^{(l)}), \quad\text{for } k<l,\text{ and}\quad L_q(\bs\pi^{(k)})\otimes L_q(\widetilde{\bs\pi}^{(l)}) \quad\text{for all }k,l.
	\end{equation*}
	Moreover, if all these tensor products are irreducible, then so is $L_q(\bs{\pi})\otimes L_q(\widetilde{\bs{\pi}})$.\qed
\end{cor}

Although the following two corollaries of \Cref{l:hlwquot} will not be used here, they reveal interesting properties of pseudo $q$-fatcorization graphs afforded by trees. Since trees are the main topic of the present paper, we find it fit to include them here for record keeping. Let us use the following terminology. 	If $G$ and $G'$ are pseudo $q$-factorization graphs over $\bs\pi$ and $\bs\pi'$, respectively, we shall say  the ordered pair $(G,G')$ is highest-$\ell$-weight if so is $L_q(\bs\pi)\otimes L_q(\bs\pi')$. We shall say $G\otimes G'$ is highest-$\ell$-weight if either $(G,G')$ or $(G',G)$ is highest-$\ell$-weight.

\begin{cor}\label{c:roexists}
	Suppose $G$ is a pseudo $q$-fatcorization graph and $(G',G'')$ is a connected cut of $G$. If $G$ is a tree, then  $G'\otimes G''$ is highest-$\ell$-weight.
\end{cor} 

\begin{proof}
		Since $G$ is a tree and both $G'$ and $G''$ are connected, there exist unique $v'\in G'$ and $v''\in G''$ such that $d(v',v'')=1$ (\Cref{l:elemtree}). In particular, 
	\begin{equation*}
		w'\in\mathcal V_{G'}, \ w''\in\mathcal V_{G''}, \ (w',w'')\ne (v',v'') \quad\Rightarrow\quad L_q(w')\otimes L_q(w'') \text{ is simple.}
	\end{equation*}
	If $(v',v'')\in\mathcal A_G$, we have $L_q(v')\otimes L_q(v'')$ is highest-$\ell$-weight and, hence, \Cref{l:hlwquot} implies so is $(G',G'')$. Otherwise, $(v'',v')\in\mathcal A_G$ and it follows that $(G'',G')$  is highest-$\ell$-weight.	
\end{proof}

Having the last proof in mind, the following is immediate from \Cref{c:vnvstar}. 
\begin{cor}\label{c:sufoneorder}
	Suppose $G$ is a pseudo $q$-fatcorization and $(G',G'')$ is a connected cut of $G$. If $G$ is a tree and $v'\in G', v''\in G''$ are such that $(v',v'')\in\mathcal A_G$, then $G'\otimes G''$ is simple if and only if $(G'',G')$ is highest-$\ell$-weight.\qed
\end{cor} 

The last result from \cite{mosi} we shall need is: 

\begin{prop}[{\cite[Corollary 4.3.3]{mosi}}]\label{c:source<-sink} Let $\bs{\pi}',\bs{\pi}''\in\mathcal{P}^+$ with dissociate factorizations and $\bs{\pi}=\bs{\pi}'\bs{\pi}''$. Let also 
	$G=G(\bs{\pi}), G'=G(\bs{\pi}'), G''=G(\bs{\pi}'')$, and suppose $\bs{\omega}',\bs{\omega}''\in\mathcal{P}^+$ satisfy 
	\begin{equation*}
		\bs{\omega}' \text{ is a source in } G', \quad \bs{\omega}'' \text{ is a sink in } G'', \quad\text{and}\quad (\bs{\omega}'',\bs{\omega}') \in \mathcal{A}_G.
	\end{equation*}
	Then, $L_q(\bs{\pi}')\otimes L_q(\bs{\pi}'')$ is not highest-$\ell$-weight.\qed
\end{prop}

\subsection{Some Reducibility Criteria}\label{ss:redc}
Recall the main result of \cite{Hernandez1} implies the determination of the simplicity of tensor products can be reduced to that of two-fold tensor products.

\begin{thm}\label{irred}
	If $S_1,\cdots, S_n$ are simple $U_q(\tlie g)$-modules, the tensor product $$S_1\otimes\cdots\otimes S_n$$ is simple if, and only if, $S_i\otimes S_j$ is simple for all $1\leq i<j\leq n$.\hfil\qed
\end{thm}

Next, we prove a couple of criteria for deciding if a tensor product is reducible which will be used in the proofs of \Cref{p:critsimp} and \Cref{t:3lineprime}.

\begin{prop}\label{p:critredp}
	Suppose $\bs{\pi},\bs{\pi}^{\prime},\bs{\pi}^{\prime\prime}\in\mathcal{P}^+$ are such that the tensor products 
	\begin{equation*}
		V=L_q(\bs{\pi})\otimes L_q(\bs{\pi}^{\prime})\quad\textrm{and}\quad V'=L_q(\bs{\pi}^{\prime})\otimes L_q(\bs{\pi}^{\prime\prime})
	\end{equation*}
	are highest-$\ell$-weight. 
	Then, $L_q(\bs{\pi})\otimes L_q(\bs{\pi}^{\prime}\bs{\pi}^{\prime\prime})$  is reducible provided $V$ is reducible and similarly for $L_q(\bs{\pi}\bs{\pi}^{\prime})\otimes L_q(\bs{\pi}^{\prime\prime})$ if $V'$ is reducible.
\end{prop}

\begin{proof}
	The second claim follows from the first by duality arguments, so we focus on the first. Since $V$ is reducible, there exists $\bs{\lambda}\in\mathcal{P}^+$ for which there exists a  non-surjective monomorphism $f:L_q(\bs{\lambda})\hookrightarrow V$. 
	In particular,  we must have
	\begin{equation}\label{eq:ineqlp}
		\operatorname{wt}(\bs{\lambda})<\operatorname{wt}(\bs{\pi})+\operatorname{wt}(\bs{\pi}^{\prime}). 
	\end{equation}
	We also consider the monomorphism:
	\begin{equation*}
		f\otimes\operatorname{id}_{L_q(\bs{\pi}^{\prime\prime})}:L_q(\bs{\lambda})\otimes L_q(\bs{\pi}^{\prime\prime})\hookrightarrow L_q(\bs{\pi})\otimes L_q(\bs{\pi}^{\prime})\otimes L_q(\bs{\pi}^{\prime\prime}).
	\end{equation*}
	On the other hand, since $V'$ is highest-$\ell$-weight, we have epimorphisms
	\begin{equation*}
		g: V'\twoheadrightarrow L_q(\bs{\pi}^{\prime}\bs{\pi}^{\prime\prime})
	\end{equation*}
	and
	\begin{equation*}
		\operatorname{id}_{L_q(\bs{\pi})}\otimes g:L_q(\bs{\pi})\otimes L_q(\bs{\pi}^{\prime})\otimes L_q(\bs{\pi}^{\prime\prime})\twoheadrightarrow L_q(\bs{\pi})\otimes L_q(\bs{\pi}^{\prime}\bs{\pi}^{\prime\prime}):=W.
	\end{equation*}
	
	By  \Cref{c:nonzeromorph}, the composition
	\begin{equation*}
		\varphi:L_q(\bs{\lambda})\otimes L_q(\bs{\pi}^{\prime\prime})\xrightarrow{f\otimes\operatorname{id}_{L_q(\bs{\pi}^{\prime\prime})}} L_q(\bs{\pi})\otimes L_q(\bs{\pi}^{\prime})\otimes L_q(\bs{\pi}^{\prime\prime})\xrightarrow{\operatorname{id}_{L_q(\bs{\pi})}\otimes g}W
	\end{equation*}
	is non-zero. Therefore, if $W$ were simple, $\varphi$ would be surjective. However, by \eqref{eq:ineqlp}, we have
	\begin{equation*}
		\operatorname{wt}(\bs{\lambda})+\operatorname{wt}({\bs{\pi}^{\prime\prime}})<\operatorname{wt}(\bs{\pi})+\operatorname{wt}(\bs{\pi}^{\prime})+\operatorname{wt}(\bs{\pi}^{\prime\prime}),
	\end{equation*}
	yielding a contradiction. Thus, $W$ is reducible as claimed.
\end{proof}

\begin{prop}\label{p:critred2p}
	Let $\bs{\pi},\bs{\pi}^{\prime},\bs{\pi}^{\prime\prime}\in\mathcal{P}^+$ such that the tensor products
	\begin{equation*}
		L_q(\bs{\pi})\otimes L_q(\bs{\pi}^{\prime})\quad\textrm{and}\quad L_q(\bs{\pi}^{\prime\prime})\otimes L_q(\bs{\pi}^{\prime})
	\end{equation*}	
	are highest-$\ell$-weight and
	$$L_q(\bs{\pi}^{\prime\prime})\otimes L_q(\bs{\pi})^*$$ is simple.  Then, if $L_q(\bs{\pi})\otimes L_q(\bs{\pi}^{\prime})$ is reducible, so is $$L_q(\bs{\pi})\otimes L_q(\bs{\pi}^{\prime}\bs{\pi}^{\prime\prime}).$$ 
	Similarly, if 
	\begin{equation*}
		L_q(\bs{\pi}')\otimes L_q(\bs{\pi})\quad\textrm{and}\quad L_q(\bs{\pi}')\otimes L_q(\bs{\pi}'')
	\end{equation*}	
	are highest-$\ell$-weight, with the former reducible, and $L_q(\bs{\pi}^{\prime\prime})^*\otimes L_q(\bs{\pi})$ is simple, then $L_q(\bs{\pi})\otimes L_q(\bs{\pi}^{\prime}\bs{\pi}^{\prime\prime})$ is reducible. 
\end{prop}

\begin{proof}
	Te second statement follows from the first by  duality arguments. To shorten notation for proving the first, set 
	\begin{equation*}
		V=L_q(\bs{\pi})\otimes L_q(\bs{\pi}^{\prime}), \quad U=L_q(\bs{\pi}^{\prime\prime})\otimes L_q(\bs{\pi}^{\prime}), \quad\textrm{and}\quad M=L_q(\bs{\pi}^{\prime\prime})\otimes L_q(\bs{\pi})^*.
	\end{equation*}	
	As in the proof of \Cref{p:critredp}, we consider the monomorphism $f$ and \eqref{eq:ineqlp} remains valid.
	
	By \eqref{e:frobrec}, we have a non-zero homomorphism 
	\begin{equation*}
		g:L_q(\bs{\pi})^*\otimes L_q(\bs{\lambda})\twoheadrightarrow L_q(\bs{\pi}^{\prime})
	\end{equation*}
	which is surjective, since $L_q(\bs{\pi}^{\prime})$ is simple. 	Therefore, we also have an epimorphism
	\begin{equation*}
		\operatorname{id}_{L_q(\bs{\pi}^{\prime\prime})}\otimes g:L_q(\bs{\pi}^{\prime\prime})\otimes L_q(\bs{\pi})^*\otimes L_q(\bs{\lambda})\twoheadrightarrow L_q(\bs{\pi}^{\prime\prime})\otimes L_q(\bs{\pi}^{\prime}).
	\end{equation*}
	In its turn, since $U$ is highest-$\ell$-weight, we have an epimorphism
	\begin{equation*}
		g^{\prime}:U\twoheadrightarrow L_q(\bs{\pi}^{\prime}\bs{\pi}^{\prime\prime})
	\end{equation*}
	and the following composition is obviously nonzero (hence, surjective):
	\begin{equation*}
		L_q(\bs{\pi}^{\prime\prime})\otimes L_q(\bs{\pi})^*\otimes L_q(\bs{\lambda})\xrightarrow{\operatorname{id}_{L_q(\bs{\pi}^{\prime\prime})}\otimes g} L_q(\bs{\pi}^{\prime\prime})\otimes L_q(\bs{\pi}^{\prime})\xrightarrow{g^{\prime}}L_q(\bs{\pi}^{\prime}\bs{\pi}^{\prime\prime}).
	\end{equation*}

	Since $M$ is simple, we have isomorphisms
	\begin{equation*}
		h:L_q(\bs{\pi})^*\otimes L_q(\bs{\pi}^{\prime\prime})\rightarrow M
	\end{equation*}
	and
	\begin{equation*}
		h\otimes \operatorname{id}_{L_q(\bs{\lambda})}:L_q(\bs{\pi})^*\otimes L_q(\bs{\pi}^{\prime\prime})\otimes L_q(\bs{\lambda})\rightarrow M\otimes L_q(\bs{\lambda}).
	\end{equation*}
	It follows that the following composition is an epimorphism:
	\begin{equation*}
		L_q(\bs{\pi})^*\otimes L_q(\bs{\pi}^{\prime\prime})\otimes L_q(\bs{\lambda})\xrightarrow{h\otimes \operatorname{id}_{L_q(\bs{\lambda})}} L_q(\bs{\pi}^{\prime\prime})\otimes L_q(\bs{\pi})^*\otimes L_q(\bs{\lambda})\xrightarrow{g^{\prime}\circ(\operatorname{id}_{L_q(\bs{\pi}^{\prime\prime})}\otimes g)}L_q(\bs{\pi}^{\prime}\bs{\pi}^{\prime\prime}).
	\end{equation*}
	By \eqref{e:frobrec}, we obtain a non-zero homomorphism
	\begin{equation*}
		\varphi:L_q(\bs{\pi}^{\prime\prime})\otimes L_q(\bs{\lambda})\rightarrow L_q(\bs{\pi})\otimes L_q(\bs{\pi}^{\prime}\bs{\pi}^{\prime\prime}).
	\end{equation*}
	The same argument used at the end of the proof of \Cref{p:critredp}  using  \eqref{eq:ineqlp} completes the proof here as well.
\end{proof}

We end this subsection with:

\begin{proof}[Proof of \Cref{p:critsimp}]
	Since $\bs\omega'\prec\bs\omega$, it follows from  \Cref{l:hlwquot} that $L_q(\bs\pi)\otimes L_q(\bs\pi')$ is highest-$\ell$-weight. Therefore, by  \Cref{c:vnvstar}, it suffices to prove that $L_q(\bs\pi')\otimes L_q(\bs\pi)$ is also highest-$\ell$-weight. Since the proofs of both situations are similar, we only write down the proof in the case condition (i) is satisfied. Thus, $G(\bs\pi\bs\pi')$ contains the following subgraph 
	\begin{tikzcd}
		\widetilde{\bs\omega} &  \arrow[swap,l] \bs\omega \arrow[r] & \bs{\omega'}
	\end{tikzcd}.
	
	To shorten notation, write $G=G(\bs\pi), G'=G(\bs\pi'), \mathcal V=\mathcal V_G$, and $\mathcal V'=\mathcal V_{G'}$. Consider the partition
	\begin{equation*}
		\mathcal V= \mathcal V_+ \mathbin{\dot{\cup}} \{\bs\omega,\widetilde{\bs\omega}\}\mathbin{\dot{\cup}}\mathcal V_-
	\end{equation*}
	where, for each $\bs\varpi\in\mathcal V_G\setminus\{\bs\omega,\widetilde{\bs\omega}\}$, we have
	\begin{equation*}
		\bs\varpi\in \mathcal V_+  \quad\Leftrightarrow\quad \text{either}\quad [\bs\varpi,\bs\omega]\cap\mathcal A_{\bs\omega}^1\ne\emptyset \quad\text{or}\quad [\bs\varpi,\widetilde{\bs\omega}] \cap\mathcal A_{\widetilde{\bs\omega}}^1\ne\emptyset \quad\text{and}\quad  \bs\omega\notin[\bs\varpi,\widetilde{\bs\omega}].
	\end{equation*}
	Thus,  
	\begin{equation*}
		\bs\varpi\in \mathcal V_-  \quad\Leftrightarrow\quad  \text{either}\quad [\bs\varpi,\widetilde{\bs\omega}]\cap\mathcal A_{\widetilde{\bs\omega}}^{-1}\ne\emptyset \quad\text{or}\quad [\bs\varpi,\bs\omega] \cap\mathcal A_{\bs\omega}^{-1}\ne\emptyset \quad\text{and}\quad  \widetilde{\bs\omega}\notin[\bs\varpi,\bs\omega].
	\end{equation*}
	Consider also the partition $\mathcal V'= \mathcal V'_+ \mathbin{\dot{\cup}} \{\bs\omega'\}\mathbin{\dot{\cup}}\mathcal V'_-$ with
	\begin{equation*}
		\bs\varpi\in \mathcal V'_+  \quad\Leftrightarrow\quad [\bs\varpi,\bs\omega']\cap \mathcal A_{\bs\omega'}^1\ne \emptyset\qquad\text{and}\qquad \bs\varpi\in \mathcal V'_-  \quad\Leftrightarrow\quad [\bs\varpi,\bs\omega']\cap \mathcal A_{\bs\omega'}^{-1}\ne \emptyset.
	\end{equation*}
	Set
	\begin{equation*}
		\bs\pi_\pm = \prod_{\bs\varpi\in\mathcal V_\pm} \bs\varpi, \qquad  \bs\pi'_\pm = \prod_{\bs\varpi\in\mathcal V'_\pm} \bs\varpi,
	\end{equation*}
	\begin{equation*}
		W=  L_q(\bs\pi_+)\otimes L_q(\bs\omega\widetilde{\bs\omega})\otimes L_q(\bs\pi_-), \qquad\text{and}\qquad W'=L_q(\bs\pi'_+)\otimes L_q(\bs\omega')\otimes L_q(\bs\pi'_-).
	\end{equation*}
	We will show that $W,W'$, and $W'\otimes W$ are highest-$\ell$-weight. In particular, there exists a surjective map $W'\otimes W\to L_q(\bs\pi')\otimes L_q(\bs\pi)$, which completes the proof of the proposition.
	
	By definition of $\mathcal V^+$, 
	\begin{equation*}
		L_q(\bs\varpi)\otimes L_q(\bs\omega) \quad\text{and}\quad L_q(\bs\varpi)\otimes L_q(\widetilde{\bs\omega}) \quad\text{are highest-$\ell$-weight for all}\quad \bs\varpi\in\mathcal V_+.
	\end{equation*}
	Indeed, at most one among $(\bs\varpi,\bs\omega)$ and $(\bs\varpi,\widetilde{\bs\omega})$ is an arrow in $G$ while neither $(\bs\omega,\bs\varpi)$ nor $(\widetilde{\bs\omega},\bs\varpi)$ is an arrow in $G$. Then,  \Cref{l:hlwquot} implies 
	\begin{equation*}
		L_q(\bs\pi_+)\otimes L_q(\bs\omega\widetilde{\bs\omega}) \qquad\text{is highest-$\ell$-weight.}
	\end{equation*}
	A similar argument shows that
	\begin{equation*}
		L_q(\bs\omega\widetilde{\bs\omega})\otimes L_q(\bs\pi_-), \quad L_q(\bs\pi'_+)\otimes L_q(\bs\omega'), \quad\text{and}\quad L_q(\bs\omega')\otimes L_q(\bs\pi'_-)
	\end{equation*}
	are highest-$\ell$-weight. Evidently, all $q$-factors of $\bs\pi_\pm$ are not adjacent to any $q$-factor of $\bs\pi'$ and, therefore,
	\begin{equation*}
		L_q(\bs\pi_\pm)\otimes L_q(\bs\varpi) \quad\text{is highest-$\ell$-weight for}\quad \bs\varpi\in\{\bs\pi'_+,\bs\pi'_-,\bs\omega'\}.
	\end{equation*}
	Similarly,
	\begin{equation*}
		L_q(\bs\varpi)\otimes L_q(\bs\pi'_\pm) \quad\text{is highest-$\ell$-weight for}\quad \bs\varpi\in\{\bs\pi_+,\bs\pi_-,\bs\omega\widetilde{\bs\omega}\}.
	\end{equation*}
	By assumption of condition (i), $L_q(\bs\omega\widetilde{\bs\omega})\otimes L_q(\bs\omega')$ is simple. In light of  \Cref{cyc}, it remains to check that
	\begin{equation*}
		L_q(\bs\pi_+)\otimes L_q(\bs\pi_-) \quad\text{and}\quad L_q(\bs\pi'_+)\otimes L_q(\bs\pi'_-) \quad\text{are highest-$\ell$-weight.}
	\end{equation*}	
	This follows by observing that $G(\bs\pi_+)$ and $G(\bs\pi_-)$ belong to different connected components of $G(\bs\pi_+\bs\pi_-)$  and similarly for $G(\bs\pi'_+\bs\pi'_-)$ (cf.  \cite[Lemma 2.2.2]{mosi}(e)). 
\end{proof}

\subsection{Proof of \Cref{c:nonadcjprim}}\label{ss:critprimetree}
Begin by recalling the following elementary property of trees.

\begin{lem}\label{l:elemtree}
	Assume $G$ is a tree.  If $\mathcal V_1\cup\mathcal V_2$ is a nontrivial partition of $\mathcal V_G$ such that $G_{\mathcal V_i}$ is connected for $i=1,2$, there exists unique $(v_1,v_2)\in\mathcal V_1\times\mathcal V_2$ such that $d(v_1,v_2)=1$. \qed
\end{lem}

\Cref{c:nonadcjprim} can be easily deduced from the next lemma together with \Cref{l:elemtree}.

\begin{lem}\label{p:critred2}
	Let $\bs{\pi},\bs{\pi}'\in\mathcal{P}^+$ have dissociate $q$-factorizations and set $G=G(\bs{\pi}), G'=G(\bs{\pi}')$.
	Assume $G$ and $G'$ are simply linked by $(\bs{\omega},\bs{\omega}')$ in $G\otimes G'$ with  $\bs{\omega}\in\mathcal{V}_G$ and $\bs{\omega}'\in\mathcal{V}_{G'}$
	and that 
	\begin{equation*}
		L_q(\bs\varpi)^*\otimes L_q(\bs\varpi') \quad\text{is simple for all}\quad \bs\varpi\in\mathcal{V}_{G},\bs\varpi'\in\mathcal{V}_{G'} \quad\text{such that}\quad \bs\varpi\ne\bs{\omega}\text{ and }\bs\varpi'\ne\bs{\omega}'.
	\end{equation*}
	Then, $L_q(\bs{\pi})\otimes L_q(\bs{\pi}')$ is reducible.
\end{lem}

\begin{proof}
	We will prove by induction on the number $N\geq 2$ of vertices of $G(\bs\pi\bs\pi')$ that there exists a nonzero homomorphism 
	\begin{equation*}
		f:V\rightarrow L_q(\bs{\pi})\otimes L_q(\bs{\pi}')
	\end{equation*}
	for some finite-dimensional $U_q(\tilde{\mathfrak{g}})$-module $V$ such that
	\begin{equation}\label{e:critred2}
		\lambda<\wt(\bs{\pi})+\wt(\bs{\pi}') \qquad\text{for all}\qquad \lambda\in\wt(V).
	\end{equation}
	This obviously implies $L_q(\bs{\pi})\otimes L_q(\bs{\pi}')$ is reducible. 
	
	The assumption that $G$ and $G'$ are simply linked together with \Cref{l:hlwquot} implies $L_q(\bs{\pi})\otimes L_q(\bs{\pi}')$ is highest-$\ell$-weight. If $N=2$, then $\bs{\pi}=\bs{\omega}, \bs\pi'=\bs\omega'$ and, since $(\bs\omega,\bs\omega')\in\mathcal{A}_{G(\bs\pi\bs\pi')}$, $L_q(\bs{\pi})\otimes L_q(\bs{\pi}')$ is a reducible highest-$\ell$-weight module, showing that inductions starts.

	For $N>2$, either $G$ or $G'$ has an extra vertex. In light of \eqref{e:dualgraphs}, up to replacing $G(\bs\pi\bs\pi')$ by its arrow-dual, we can assume without loss of generality that $G$ has an extra vertex. Moreover, since each connected component of $G$  has a sink and a source, we can choose either a sink or a source among the extra vertices. Henceforth, let $\bs\varpi$ be such a choice. 
	
	Quite clearly, after replacing $\bs\pi$ with $\bs\pi\bs\varpi^{-1}$, the hypotheses of the lemma remain valid. Thus, the inductive hypothesis implies there exists and a non-zero homomorphism
	\begin{equation*}
		h:W\rightarrow L_q(\bs{\pi}\bs\varpi^{-1})\otimes L_q(\bs{\pi}'),
	\end{equation*}
	where $W$ is a finite-dimensional module such that
	\begin{equation}\label{e:critred2i}
		\lambda<\operatorname{wt}(\bs{\pi})-\operatorname{wt}(\bs\varpi)+\operatorname{wt}(\bs{\pi}')
	\end{equation}
	for all $\lambda\in\operatorname{wt}(W)$. 
	
	Assume first that $\bs\varpi$ is a source in $G$. In that case,  $L_q(\bs\varpi)\otimes L_q(\bs{\pi}\bs\varpi^{-1})$ is highest-$\ell$-weight by  \Cref{l:hlwquot} and we have an epimorphism
	\begin{equation*}
		p:L_q(\bs\varpi)\otimes L_q(\bs{\pi}\bs\varpi^{-1})\twoheadrightarrow L_q(\bs{\pi}).
	\end{equation*}
	Let $V=L_q(\bs\varpi)\otimes W$ and let $f$ be the composition
	\begin{equation*}
		L_q(\bs\varpi)\otimes W\xrightarrow{\operatorname{id}_{L_q(\bs\varpi)}\otimes h}L_q(\bs\varpi)\otimes L_q(\bs{\pi}\bs\varpi^{-1})\otimes L_q(\bs{\pi}')\xrightarrow{p\otimes\operatorname{id}_{L_q(\bs{\pi}')}} L_q(\bs{\pi})\otimes L_q(\bs{\pi}').
	\end{equation*}
	\Cref{c:nonzeromorph} implies $f$ is non-zero. Since $\ch(V) = \ch(L_q(\bs\varpi))\ch(W)$, we have
	\begin{equation*}
		\mu\in\wt(V) \quad\Rightarrow\quad \mu=\nu+\lambda \quad\text{with}\quad \nu\le\wt(\bs\varpi), \lambda\in\wt(W).
	\end{equation*}
	Thus, \eqref{e:critred2} follows from \eqref{e:critred2i}.

	Finally, if $\bs\varpi$ is a sink in $G$, we set $V=W\otimes L_q(\bs\varpi)$. Evidently, \eqref{e:critred2} follows from \eqref{e:critred2i} as before and, therefore, we need to show there exists a nonzero map $f:V\to L_q(\bs{\pi})\otimes L_q(\bs{\pi}')$.

	\Cref{l:hlwquot} implies $L_q(\bs{\pi}^{(1)}\bs\varpi^{-1})\otimes L_q(\bs\varpi)$ is highest-$\ell$-weight and, hence, we have an epimorphism
	\begin{equation*}
		p:L_q(\bs{\pi}\bs\varpi^{-1})\otimes L_q(\bs\varpi)\twoheadrightarrow L_q(\bs{\pi}).
	\end{equation*} 
	Since $h\ne 0$, \eqref{e:frobrec} implies we have a non-zero homomorphism
	\begin{equation*}
		g:W\otimes ~^*L_q(\bs{\pi}')\rightarrow L_q(\bs{\pi}\bs\varpi^{-1}).
	\end{equation*}
	In particular, since $L_q(\bs{\pi}\bs\varpi^{-1})$ is simple, $g$ is surjective and, hence, so are
	\begin{equation*}
		g\otimes\operatorname{id}_{L_q(\bs\varpi)}:W\otimes ~^*L_q(\bs{\pi}')\otimes L_q(\bs\varpi)\rightarrow L_q(\bs{\pi}\bs\varpi^{-1})\otimes L_q(\bs\varpi)
	\end{equation*}
	and the composition
	\begin{equation*}
		W\otimes ~^*L_q(\bs{\pi}')\otimes L_q(\bs\varpi)\xrightarrow{g\otimes\operatorname{id}_{L_q(\bs\varpi)}} L_q(\bs{\pi}\bs\varpi^{-1})\otimes L_q(\bs\varpi)\xrightarrow{p}L_q(\bs{\pi}).
	\end{equation*}
	Since $\bs\varpi\neq\bs\omega$, it follows from our assumptions that $~^*L_q(\bs{\pi}')\otimes L_q(\bs\varpi)$ is simple and we have an isomorphism
	\begin{equation*}
		\varphi: V\otimes  ~^*L_q(\bs{\pi}') = W\otimes L_q(\bs\varpi)\otimes ~^*L_q(\bs{\pi}')\rightarrow W\otimes ~^*L_q(\bs{\pi}')\otimes L_q(\bs\varpi).
	\end{equation*}
	Thus, the composition 
	\begin{equation*}
		V\otimes ~^*L_q(\bs{\pi}')\xrightarrow{\varphi}W\otimes ~^*L_q(\bs{\pi}')\otimes L_q(\bs\varpi)\xrightarrow{g\otimes\operatorname{id}_{L_q(\bs\varpi)}} L_q(\bs{\pi}\bs\varpi^{-1})\otimes L_q(\bs\varpi)\xrightarrow{p}L_q(\bs{\pi}) 
	\end{equation*}
	is surjective. Finally, applying \eqref{e:frobrec} again we obtain a non-zero homomorphism
	\begin{equation*}
		V\rightarrow L_q(\bs{\pi})\otimes L_q(\bs{\pi}')
	\end{equation*}
	as desired.
\end{proof}

\section{Three-Vertex Prime Graphs}\label{s:3primes}

This section is dedicated to the proofs of Theorems \ref{t:3lineprime} and \ref{t:realtree}. In fact, most of the section is concerned with the former, while the latter is treated in the final (short) subsection.

\subsection{Tensor Product of Fundamental Modules} \label{soctpkr}

Assume $\lie g$ is of type $A$ and identify $I$ with the set $\{1,\dots,n\}$ in such a way that $d(i,j)=|j-i|$ for all $i,j\in I$ (there are two options). Fix $i,j\in I, a\in\mathbb F^\times, r,s\in\mathbb Z_{> 0}, m\in\mathscr R_{i,j}^{r,s}$, and consider
\begin{equation*}
	V= L_q(\bs\omega_{j,aq^m,s})\otimes L_q(\bs\omega_{i,a,r}).
\end{equation*}
One can consider the problem of characterizing  the simple factors of $V$. For the purpose of proving \Cref{t:3lineprime}, we need the case $r=s=1$ only, in which case
\begin{equation}
	m = 2+d(i,j)-2p \quad\text{for some}\quad -d([i,j],\partial I)\le p\le 0.
\end{equation}
Letting $i_-=\min\{i,j\}, i_+=\max\{i,j\}$, and $J$ the minimal connected subdiagram such that $m\in \mathscr R_{i,j,J}^{1,1}$, we have
\begin{equation*}
	J = [i_-+p,i_+-p].
\end{equation*}
Evidently, if $L_q(\bs\mu)$ is such a simple factor, we must have 
\begin{equation*}
	\bs\mu\in D := \wtl(V)\cap\mathcal P^+.
\end{equation*}

The qcharacter of the fundamental modules are well-known and can be described in terms of column tableaux (see \cite[Theorems 3.8 and 3.10]{her:min}, \cite[Corollary 7.6 and Remark 7.4 (i)]{muyou:path},  and references therein). We partially review this approach. 

Given $a\in\mathbb F^\times$, let $\ffbox{i}_{\:s}$ denote the $\ell$-weight $\bs\omega_{i,aq^{s+i-1}}\bs\omega_{i-1,aq^{s+i}}^{-1}, 1\le i\le n+1$. Here, we understand $\bs\omega_{0,b}=\bs 1=\bs\omega_{n+1,b}$. The index $s$ of the box will be referred to as its content while the index $s$ as its support. A semi-standard column tableau of height $k$ supported at $s$ is a vertical juxtaposition of boxes
\begin{equation*}
	\renewcommand{\arraycolsep}{0pt}
	\begin{array}{|c|}
		\hline \hbox to 0.55cm{\hfill$\scriptstyle i_1$\hfill} \\
		\hline \hbox to 0.55cm{\hfill$\vdots$\hfill} \\
		\hline \hbox to 0.55cm{\hfill$\scriptstyle{i_k}$\hfill} \\
		\hline
	\end{array}_{\:s}  \quad\text{with}\quad 1\le i_1\le \cdots\le i_k\le n+1
\end{equation*} 
and the support of the box having content $i_j$ is $s+2(k-j)$. The $\ell$-weight associated to such tableau $T$ is
\begin{equation*}
	\bs\omega^T = \prod_{j=1}^k  \ffbox{i_j}_{\:s+2(k-j)}.
\end{equation*}
In particular, if $i_j=j$ for all $1\le j\le k$, $\bs\omega^T = \bs\omega_{k,aq^{s+k-1}}$. Let ${\rm Tab}_{i,s}$ denote the set of all semi-standard column tableaux of height $i$ supported at $1-i$. Then,
\begin{equation}
	\qch(L_q(\bs\omega_{i,a})) = \sum_{T\in{\rm Tab}_{i,s}} \bs\omega^T.
\end{equation}                                                                                       
In particular, 
\begin{equation*}
	\dim(L_q(\bs\omega_{i,a})_{\bs\varpi}) =  1 \quad\text{for all}\quad \bs\varpi\in\wtl(L_q(\bs\omega_{i,a})).
\end{equation*}

Proceeding as in \cite[Lemma 3.3.1]{mope:tpa}, one easily checks that every element of $D$ is of the form 
\begin{equation}\label{e:toplow}
	\bs\nu\bs\omega_{j,aq^m} \quad\text{with}\quad \bs\nu\in\wtl(L_q(\bs\omega_{i,a}))
\end{equation}
and, therefore, we must have $\bs\nu=\bs\omega^T$ for some $T\in{\rm Tab}_{i,s}$ having at most one gap. Note $T$ has exactly one gap if and only if it is of the form
\begin{equation*}
	T_{k,l}: = \renewcommand{\arraycolsep}{0pt}
	\begin{array}{|c|}
		\hline \hbox to 0.55cm{\hfill$\scriptstyle 1$\hfill} \\
		\hline \hbox to 0.55cm{\hfill$\vdots$\hfill} \\
		\hline \hbox to 0.55cm{\hfill$\scriptstyle k$\hfill} \\
		\hline \hbox to 0.55cm{\hfill$\scriptstyle l+1$\hfill} \\
		\hline \hbox to 0.55cm{\hfill$\vdots$\hfill} \\
		\hline \hbox to 1.1cm{\hfill$\scriptstyle{l+i-k}$\hfill} \\
		\hline
	\end{array}_{\:1-i} \quad\text{with}\quad 0\le k <\min\{i,l\}, \ l\le n-i+k+1.
\end{equation*} 
Note also
\begin{equation}\label{e:qcfund}
	\bs\omega^{T_{k,l}} = \bs\omega_{l,aq^{i+l-2k}}^{-1}\ \bs\omega_{k,aq^{i-k}}\ \bs\omega_{i+l-k,aq^{l-k}}.
\end{equation}
Letting $J_{k,l} = [k+1,l+i-k-1]$, one can easily check that (cf. \cite[Eq. (2.2.11)]{mope:tpa})
\begin{equation}\label{e:tpfroots}
	\omega_i - \wt(\bs\omega^{T_{k,l}})\in Q_{J_{k,l}}^+
\end{equation}

Evidently, for $\bs\nu=\bs\omega_{i,a}, L_q(\bs\nu\bs\omega_{j,aq^m})$ is a simple factor. On the other hand, if $\bs\nu$ is as in \eqref{e:qcfund}, then $\bs\nu\bs\omega_{j,aaq^m}\in\mathcal P^+$ if and only if  $l=j,i+l-2k = m$, and, hence,
\begin{equation}\label{e:tpfroots=min}
	k = p-1 + i_-, \qquad J_{k,l} = J,
\end{equation}
and
\begin{equation*}
	D = \{\bs\omega_{i,a}\ \bs\omega_{j,aq^m}\ ,\ \bs\omega_{i_-+p-1,aq^{1-p+i-i_-}}\ \bs\omega_{i_++1-p,aq^{1-p+j-i_-}}  \}.
\end{equation*} 
Moreover, 
\begin{equation}\label{soctpd}
	L_q(\bs\omega_{i_-+p-1,aq^{1-p+i-i_-}})\otimes L_q(\bs\omega_{i_++1-p,aq^{1-p+j-i_-}}) \quad\text{is simple}. 
\end{equation}
Indeed, to check this, it suffices to show that
\begin{equation*}
	d(i,j)=|(1-p+j-i_-) - (1-p+i-i_-)|\notin\mathscr R_{i_-+p-1,i_++1-p}^{1,1}.
\end{equation*}
The elements of $\mathscr R_{i_-+p-1,i_++1-p}^{1,1}$ are of the form $2+ d(i_-+p-1,i_++1-p) - 2p'$ with $p'\le 0$. Since
\begin{equation*}
	d(i_-+p-1,i_++1-p) = 2(1-p) + d(i,j)>d(i,j),
\end{equation*}
\eqref{soctpd} follows. In summary, we have proved there exists a short exact sequence (see also \cite{mope:tpa} and references therein).
\begin{equation}\label{socforl2}
	0\to L_q(\bs\omega_{i_-+p-1,aq^{1-p+i-i_-}})\otimes L_q(\bs\omega_{i_++1-p,aq^{1-p+j-i_-}})\to V\to L_q(\bs\omega_{i,a} \bs\omega_{j,aq^m})\to 0. 
\end{equation}

\subsection{Singleton Cuts}\label{ss:tpfund}
Any nontrivial cut of a $q$-factorization graph with three vertices has one of its subgraphs being a singleton. More generally, if \cite[Conjecture 3.5.1]{mosi} is true, the relevance of  such cuts for the theory is evident. Thus, let us start from this broader setup.

Given $\bs\pi\in\mathcal P^+, i\in I,a\in\mathbb F^\times$, and $r\in\mathbb Z_{>0}$, let
\begin{equation}\label{e:tpfund}
	V= L_q(\bs\omega_{i,a,r})\otimes L_q(\bs\pi).
\end{equation}
The goal is to describe conditions on the given parameters for deciding whether $V$ is simple or not.
By \Cref{l:hlwquot}, if $L_q(\bs\omega)\otimes L_q(\bs\omega_{i,a,r})$ is simple for every $q$-factor $\bs\omega$ of $\bs\pi$, then $V$ is simple, but this is certainly not a necessary condition. In any case, we can assume there exists $j\in I,m,s\in\mathbb Z,s>0$, such that $\bs\omega_{j,aq^m,s}$ is a $q$-factor of $\bs\pi$ and $|m|\in\mathscr R_{i,j}^{r,s}$. We can also assume $\bs\pi$ and $\bs\omega_{i,a,r}$ have dissociated $q$-factorizations, otherwise \Cref{p:qftp} implies $V$ is reducible. Evidently,
\begin{equation*}
	L_q(\bs\pi)\cong L_q(\bs\omega_{j,aq^m,s})\otimes L_q(\bs\pi\bs\omega_{j,aq^m,s}^{-1}) \qquad\Rightarrow\qquad V\text{ is reducible.}
\end{equation*}
Thus, we shall assume $L_q(\bs\omega_{j,aq^m,s})\otimes L_q(\bs\pi\bs\omega_{j,aq^m,s}^{-1})$ is reducible.

\begin{lem}\label{l:disconect}
	Let $J$ be a connected subdiagram such that $[i,j]\subseteq J$ and $|m|\in\mathscr R^{r,s}_{i,j,J}$. 
	Let  $G(\bs\varpi)$ be the connected component of $G(\bs\pi_J)$ containing $(\bs\omega_{j,aq^m,s})_J$. 	
	If $L_q((\bs\omega_{i,a,r})_J)\otimes L_q(\bs\varpi)$ is reducible, so is $V$. In particular, this is the case if $G(\bs\varpi)$ is a singleton.
\end{lem}

\begin{proof}
	Write $\bs\pi_J = \bs\varpi\bs\varpi'$. The assumption on $G(\bs\varpi)$ together with  \Cref{l:hlwquot} implies $L_q(\bs\pi_J)\cong L_q(\bs\varpi)\otimes L_q(\bs\varpi')$ and, hence,
	\begin{equation*}
		L_q((\bs\omega_{i,a,r})_J)\otimes L_q(\bs\pi_J)\cong  L_q((\bs\omega_{i,a,r})_J)\otimes L_q(\bs\varpi)\otimes L_q(\bs\varpi'),
	\end{equation*}
	is reducible. The first conclusion now follows from \Cref{c:sJs}. Note $G(\bs\varpi)$ is a singleton if and only if $\bs\varpi = (\bs\omega_{j,aq^m,s})_J$. Thus, the second conclusion follows from the first and the assumption on $m$.
\end{proof}

The next lemma is immediate from the previous.

\begin{lem}\label{l:j'inJ}
	Let $J$ be a connected subdiagram such that $[i,j]\subseteq J$ and $|m|\in\mathscr R^{r,s}_{i,j,J}$. If $L_q((\bs\omega_{j,aq^m,s})_J)\otimes L_q(\bs\omega'_J)$ is simple for every $q$-factor $\bs\omega'$ of $\bs\pi$ adjacent to $\bs\omega_{j,aq^m,s}$ in $G(\bs\pi)$, then $V$ is reducible. In particular, if $V$ is simple, there exist $j'\in J, s'\in\mathbb Z_{>0}$, and $m'\in\pm \mathscr R_{j,j',J}^{s,s'}$ such that $\bs\omega_{j',aq^{m- m'},s'}$ is a $q$-factor of $\bs\pi\bs\omega_{j,aq^m,s}^{-1}$.\hfil\qed
\end{lem}

Henceforth, we assume $\bs\pi$  has two $q$-factors, say
\begin{equation*}
	\bs\pi = \bs\omega_{j,aq^m,s}\bs\omega_{j',b,s}.
\end{equation*}
Note condition (i)  of the above lemma becomes equivalent to $j'\ne j$, 
since we are assuming $L_q(\bs\omega_{j,aq^m,s})\otimes L_q(\bs\omega_{j',b,s'})$ is reducible. This assumption also implies we must have 
\begin{equation*}
	aq^m=bq^{m'} \qquad\text{with}\qquad |m'|\in \mathscr R_{j,j'}^{s,s'}. 
\end{equation*}
We have the following possibilities for $G:=G(\bs\pi\bs\omega_{i,a,r})$: a line (directed or alternating) with $\bs\omega_{i,a,r}\in\partial G$ or a triangle (with  $\bs\omega_{i,a,r}$ extremal or not). So, the pictures, up to arrow duality, are:

\setlength{\unitlength}{.4cm}
\begin{enumerate}[(1)]
	\item \begin{tikzcd}
	\stackrel{r}{i}  \arrow[r,"|m|"] & \stackrel{s}{j} \arrow[r,"m'"] & \stackrel{s'}{j'} 
\end{tikzcd}  (which implies $m<0$ and $m'-m\notin\mathscr R_{i,j'}^{r,s'}$); 
	
	\item\label{altcase}
 \begin{tikzcd}
	\stackrel{r}{i} & \arrow[swap,l,"m"]  \stackrel{s}{j} \arrow[r,"m'"] & \stackrel{s'}{j'} 
\end{tikzcd}(which implies $m>0$ and $|m-m'|\notin\mathscr R_{i,j'}^{r,s'}$); 
	
	\item\label{triangcase} 
 \begin{tikzcd}[column sep=small,row sep=tiny]
	& \stackrel{r}{i} \arrow[dr,"m''"]& \\  \stackrel{s}{j} \arrow[ur,"m"] \arrow[rr,"m'"] & & \stackrel{s'}{j'} 
\end{tikzcd}(which implies $m>0$ and $m''= m'-m\in\mathscr R_{i,j'}^{r,s'}$); 
	
	\item \begin{tikzcd}[column sep=small,row sep=tiny]
	& \stackrel{r}{i} & \\  \stackrel{s}{j} \arrow[ur,"m"] \arrow[rr,"m'"] & &  \arrow[swap,ul,"m''"] \stackrel{s'}{j'} 
\end{tikzcd}\hspace{5pt} or\hspace{5pt} \begin{tikzcd}[column sep=small,row sep=tiny]
& \arrow[swap,dl,"-m"] \stackrel{r}{i} \arrow[dr,"m''"]& \\  \stackrel{s}{j}  \arrow[rr,"m'"] & & \stackrel{s'}{j'} 
\end{tikzcd} (with $m>0$ in the first case and $m<0$ in the second).  \vspace{20pt}
\end{enumerate}
We chose to draw the pictures corresponding to the case $m'>0$. The cases with $m'<0$ are obtained by arrow duality (which preserves the property of the studied tensor product being simple or not). In the last two cases, the arrow duality interchanges the roles of $j$ and $j'$ in the following arguments. 
The reducibility of $V$ if $G$ is as in (1) or (4) follows from \Cref{c:source<-sink}, while case (3) follows from \Cref{t:toto} for type $A$. The next subsection is dedicated to the study of case (2), thus proving \Cref{t:3lineprime}.

\subsection{The proof of \Cref{t:3lineprime}}\label{ss:3linecut} We now suppose $G$ is as in \eqref{altcase} of the previous subsection. Let $J\subseteq I$ denote a minimal connected subdiagram satisfying
\begin{equation*}
	i,j\in J \quad\text{and}\quad m\in\mathscr R_{i,j, J}^{r,s}.
\end{equation*}
One easily checks using the characterization of the sets $\mathscr R_{i,j}^{r,s}$ in \eqref{t:krredsets} that, if $\lie g$ is of type $A$, $J$ is unique. In that case, we will show that $V$ is simple if and only if
\begin{equation}\label{2gencond}
	j'\in J, \qquad m'\in{\mathscr{R}_{j,j',J}^{s,s'}}, \qquad |m-m'-\check h_J|\in\mathscr R_{w_0^J(i),j',J}^{r,s'},
\end{equation}
and
\begin{equation}\label{2exracond}
	m-m'+1\notin\mathscr{R}_{i,j',J}^{r-1,s'}  \quad\text{if}\quad r>1.
\end{equation}

We begin by showing that \eqref{2gencond} holds if $V$ is simple and $\lie g$ is simply laced. The first two conditions are immediate from \Cref{l:j'inJ}.  For the last condition, assume $j'\in J, m'\in{\mathscr{R}_{j,j'}^{s,s'}}_J$,  consider 
\begin{equation*}
	N= L_q(\bs{\omega}_{j,aq^m,s})\otimes L_q(\bs{\omega}_{i,a,r}) \quad\textrm{and}\quad N'= L_q(\bs{\omega}_{j,aq^m,s})\otimes L_q(\bs{\omega}_{j',aq^{m-m'},s'}),
\end{equation*}
and note that $N_J$ and $N'_J$ are highest-$\ell$-weight and reducible. Thus, by the second part of  \Cref{p:critred2p}, $V_J$ is reducible provided
\begin{equation*}
	L_q((\bs{\omega}_{i,a,r})_J)\otimes L_q((\bs{\omega}_{j',aq^{m-m'},s'})_J)^*
\end{equation*}
is simple or, equivalently, 
$$|m-m'-\check h_J|\notin \mathscr R_{i,w_0^J(j'),J}^{r,s'} = \mathscr R_{w_0^J(i),j',J}^{r,s'}.$$
For type $A$, we shall see below that 
\begin{equation*}
	|m-m'-\check h_J| = m'-m+\check h_J
\end{equation*}
and \eqref{2exracond} is equivalent to
\begin{equation}\label{2exracondend}
	m+r\le m'+s' + d(i,j').
\end{equation}

Henceforth, we assume \eqref{2gencond} holds and $\lie g$ is of type $A$. In that case, we show $V$ is simple if and only \eqref{2exracond} holds. Note
\begin{equation*}
	V':= L_q(\bs\pi)\otimes L_q(\bs\omega_{i,a,r})
\end{equation*}
is highest-$\ell$-weight. Therefore, in light of \Cref{c:vnvstar},
\begin{equation}\label{Vishw}
	V \text{ is simple}\quad\Leftrightarrow\quad V\text{ is highest-$\ell$-weight.}
\end{equation}

Since $\lie g$ is of type $A$, \eqref{t:krredsets} implies
\begin{equation}\label{e:2tp1Am}
	m = r+s+d(i,j)-2p \quad\text{for some}\quad -d([i,j],\partial I)\le p<\min\{r,s\}
\end{equation}
and
\begin{equation}\label{e:2tp1Am'}
	m' = s+s'+d(j,j')-2p' \quad\text{for some}\quad -d([j,j'],\partial I)\le p'<\min\{s,s'\}.
\end{equation}
Moreover, $m\notin\mathscr R_{i,j,[i,j]}^{r,s}$ if and only if $p<0$ and similarly for $m'$.

The proof will proceed by separate analyzes of several cases and subcases. The main splitting is according the sign of the parameter $p$ defined in \eqref{e:2tp1Am}. The proof for $p\le 0$ is in \Cref{ss:p<0} while the case $p>0$ is treated in  \Cref{ss:p>0}. The knowledge of the case $p\le 0$ is used in  \Cref{ss:p>0}, so the proof of the latter is not independent of that of the former. To prove that $V$ is simple if \eqref{2exracond}  holds, in most cases we will prove that $V$ is highest-$\ell$-weight as an application of \Cref{cyc}. This will not be the case only for $p\le 0$ and $r=1$, which requires a more sophisticated argument which uses \eqref{socforl2}, \eqref{soctpd}, and, in one crucial moment, \Cref{rsh}. However, this case is used in the proof of other cases with $p\le 0$, as well as for the proof of the converse. So, indirectly, the whole proof depends on this more complicated argument for this particular subcase. To prove the converse, i.e., that $V$ is reducible if \eqref{2exracond} does not hold, we use \Cref{c:sJs} and Propositions \ref{p:critredp} and \ref{p:critred2p} for $p\le 0$, while, for $p>0$, we show $V$ is not highest-$\ell$-weight by contradiction using \Cref{cyc}.

\subsubsection{Technical Rephrasing of the Conditions}\label{ss:repcond}
We start by describing detailed conditions on the parameters $i,j,j',r,s,s',p,$ and $p'$ arising from
\eqref{2gencond} and \eqref{2exracond}. 
Recall \eqref{dikj} and that, since $\lie g$ is of type $A$, we have $d(i,\partial J) = d(j,\partial J)$, which implies
\begin{equation*}
	w_0^J(i) = j.
\end{equation*}
If $p'$ is given by  \eqref{e:2tp1Am'}, we have
\begin{equation}\label{defp=-}
	\begin{aligned}
		& m - m' = \pm ( r+s'+d(i,j') - 2p_\pm  ) \qquad\text{where}\\ p_+& = s'-p' + p + d_{i,j}^{j'} \qquad\text{and}\qquad p_- = r-p +p' + d_{j,j'}^i.
	\end{aligned}
\end{equation}
Note
\begin{gather}\notag
	|m-m'| = r+s'+d(i,j') - 2p_+ \quad\Leftrightarrow\quad 2p_+ \le r+s'+d(i,j') \quad\Leftrightarrow\quad m\ge m';
	\\ \text{and}\\ \notag
	|m-m'| = r+s'+d(i,j') - 2p_- \quad\Leftrightarrow\quad 2p_- \le r+s'+d(i,j') \quad\Leftrightarrow\quad m\le m'. 
\end{gather}
Thus, the condition $|m-m'|\notin \mathscr R_{i,j'}^{r,s'}$, arising from the assumption that $G$ is as in (2) of \Cref{ss:tpfund}, can be rephrased as
\begin{gather}\notag
	p_+ < -d([i,j'],\partial I) \quad\text{or}\quad p_+\ge \min\{r,s'\}, \qquad\text{if}\qquad m\ge m',\\ \label{notriang} \quad\text{and}\quad\\ \notag  p_- < -d([i,j'],\partial I) \quad\text{or}\quad p_-\ge \min\{r,s'\}, \qquad\text{if}\qquad m\le m'.
\end{gather}
In what follows, any statement about $p_+$ will carry implicitly the assumption that $m\ge m'$ and similarly for $p_-$.

\noindent {\bf Case 1.} 
If $p\le 0$, we have
\begin{equation}\label{J<>p}
	d(i,\partial J) = d(j,\partial J) =-p \qquad\text{and}\qquad \check h_J = d(i,j)-2p+2,
\end{equation}
and the first two requirements in \eqref{2gencond} imply
\begin{equation}\label{weird}
	-p'\le d([j,j'], \partial J)= d([i,j'], \partial J) = -p-d_{i,j}^{j'}.
\end{equation}
Indeed, one easily checks the equalities follow from \eqref{J<>p} and the condition $j'\in J$. The inequality is  immediate from the definition of $\mathscr R_{j,j',J}^{s,s'}$ and the second condition in \eqref{2gencond}. Note  \eqref{weird} implies 
\begin{equation*}
	m-m' = r-s' +2(p'-p) + d(i,j)-d(j',j) \ge r-s' + d(i,j').
\end{equation*}
Thus, 
\begin{equation}\label{m<m'=>r<s'}
	m\le m' \quad\Rightarrow\quad r\le s' - d(i,j') \le s'.
\end{equation}
Let us see that \eqref{weird} also implies 
\begin{equation}\label{2lbp+-}
	p_\pm \ge \min\{r,s'\} \qquad\text{and}\qquad p_+\le s'.
\end{equation}
In the case of $p_-$, \eqref{weird} implies
\begin{equation*}
	p_- = r+p'+d^i_{j,j'} - p \ge r+d^i_{j,j'}+ d^{j'}_{i,j}\ge r.
\end{equation*}
For $p_+$, \eqref{weird} implies
\begin{equation*}
	p_+ = s'-p'+p+d^{j'}_{i,j} \le s'.
\end{equation*}
It remais to check the first inequality for $p_+$. By \eqref{notriang}, it suffices to show we cannot have $p_+ < -d([i,j'],\partial I)$. Indeed, if this were the case, we would have
\begin{equation*}
	s'-p'+p+d^{j'}_{i,j}=p_+< -d([i,j'],\partial I)\le -d([i,j'],\partial J) = p+d_{i,j}^{j'},
\end{equation*}
which yields a contradiction since $p'<s'$. In particular, it follows from \eqref{2lbp+-} that
\begin{equation}\label{2p+=}
	r\ge s'   \quad\Rightarrow\quad p_+=s'.
\end{equation}

Next, we check the third condition in \eqref{2gencond} is equivalent to 
\begin{equation}\label{3rfP<0}
	|m-m'-\check h_J| = -(m-m'-\check h_J)\in\mathscr R_{j,j',J}^{r,s'}.
\end{equation}
To do that, it suffices to check
\begin{equation*}
	m-m'-\check h_J\notin\mathscr R_{j,j',J}^{r,s'}.
\end{equation*}
If $m-m'\le 0$, this is obvious since $m-m'-\check h_J<0$. Otherwise, 
\begin{align*}
	m - m' - h_J = r+s'+d(i,j') - 2p_+ - h_J = r+s' + d(j,j') - 2(p_+ - p + d^{j}_{i,j'} + 1)
\end{align*}
and we are done since
\begin{equation*}
	p_+ - p + d^{j}_{i,j'} + 1 \ge \min\{r,s'\} - 0 + 0 + 1 > \min\{r,s'\}.
\end{equation*}
Let us see   \eqref{3rfP<0} implies 
\begin{equation}\label{p'<ifp<}
	p'\le \min\{r,s'\}-r\le 0.
\end{equation}
Indeed, 
\begin{equation*}
	-(m-m'-\check h_J) = r+s'+d(j,j') - 2(r+p' -1),
\end{equation*}
and the claim follows from the description of $\mathscr R_{j,j',J}^{r,s'}$. This completes the information we needed to extract from  \eqref{2gencond}.

Next, for $r>1$, we check that \eqref{2exracond} is equivalent to
\begin{equation}\label{2exracond'}
	r\le s'. 
\end{equation}
If $m<m'$, then \eqref{2exracond} is obvious and \eqref{2exracond'} follows from \eqref{m<m'=>r<s'}. Otherwise, we have 
\begin{equation*}
	m-m'+1 = (r-1)+s'+d(i,j')-2(p_+-1),
\end{equation*}
and it lies in $\mathscr R_{i,j',J}^{r-1,s'}$ if and only if $-d([i,j'],\partial J)\le p_+-1<\min\{r-1,s'\}$. The first inequality in \eqref{2lbp+-} implies this is the same as
\begin{equation}\label{2exracondthink}
	\min\{r,s'\}-1  \le p_+-1<\min\{r-1,s'\},
\end{equation}
which is clearly impossible if $r\le s'$. For $r>s'$, \eqref{2p+=} implies $p_+=s'$ and, hence, $m-m'+1\in \mathscr R_{i,j',J}^{r-1,s'}$. 

Finally, let us check that \eqref{2exracond} is equivalent to \eqref{2exracondend}. Using \eqref{e:2tp1Am} and \eqref{e:2tp1Am'}, one easily checks the latter is equivalent to
\begin{equation}\label{2exracondend'}
	r-s'\le p-p' +  d^{j'}_{i,j}.
\end{equation}
This, together with \eqref{weird}, clearly implies \eqref{2exracond'} and, hence,  \eqref{2exracond}. Conversely, if \eqref{2exracond'} holds, then \eqref{2exracondend} is obvious if $m\le m'$. Otherwise, \eqref{2exracondend'} follows from the first inequality in \eqref{2lbp+-}.  In particular, \eqref{2exracondend'} holds if $r=1$.

\noindent {\bf Case 2.} If $p\ge 0$, then
\begin{equation}\label{Jp>0}
	J=[i,j] \qquad\text{and}\qquad \check h_J = d(i,j)+2
\end{equation}
and the first two conditions in \eqref{2gencond} immediately imply
\begin{equation}\label{p>0=>p'>0}
	j'\in [i,j] \qquad\text{and}\qquad p'\ge 0.
\end{equation}
Observe also that the first claim in \eqref{2lbp+-} remains valid. Indeed, in light of \eqref{notriang}, it suffices to show $p_\pm>0$. Since, by definition, $p<r$ and $p'<s'$, we have
\begin{equation*}
	p_- = (r-p)+p'+d^i_{j,j'} > 0 \qquad\text{and}\qquad p_+ = (s'-p')+p+d^{j'}_{i,j} >0.
\end{equation*}
We also observe that \eqref{3rfP<0} remains valid. As before, this is obvious if $m-m'\le 0$. Otherwise, 
\begin{align*}
	m - m' - h_J = r+s'+d(i,j') - 2p_+ - h_J = r+s' + d(j,j') - 2(p_+ + d(j,j') + 1)
\end{align*}
and we are done since
\begin{equation*}
	p_+ + d(j,j') + 1>p_+ \ge \min\{r,s'\}.
\end{equation*}
Using \eqref{defp=-}, we also have
\begin{equation*}
	-(m-m'-h_J) = r+s'+d(j,j') - 2(r-p+p'-1).
\end{equation*}
Thus,  \eqref{3rfP<0} is equivalent to
\begin{equation}\label{3rfP<0'}
	0\le r-p+p'-1< \min\{r,s'\},
\end{equation}
which, in particular, implies
\begin{equation}\label{p>0=>p>p'}
	p\ge p' \quad\text{as well as}\quad r \le s'+p-p' \ \text{ if }\ r\ge s'.
\end{equation}

Next, we check that \eqref{2exracond} is equivalent to 
\begin{equation}\label{2exracondp>0}
	r\le s' \qquad\text{or}\qquad p\ne p'.
\end{equation}
Indeed, if $m<m'$, then  \eqref{2exracond} is obvious and, if $r>s'$, the second part of \eqref{p>0=>p>p'} implies $p\ne p'$,
showing \eqref{2exracondp>0} holds. Otherwise, studying \eqref{2exracondthink}, we see
\begin{equation*}
	m-m'+1\in \mathscr R_{i,j',J}^{r-1,s'} \qquad\Leftrightarrow\qquad r> s' \quad\text{and}\quad p_+=s'.
\end{equation*}
Now \eqref{defp=-} implies the condition $p_+=s'$ is equivalent to $p=p'$, completing the proof. 

Finally, let us check that \eqref{2exracond} is equivalent to \eqref{2exracondend} or, equivalently, to  \eqref{2exracondend'} which, in this case, reduces to
\begin{equation}\label{2exracondend''}
	r-s'\le p-p'.
\end{equation}
Evidently, \eqref{2exracondend''} implies \eqref{2exracondp>0}. Conversely, \eqref{2exracondend''} clearly holds if $r\le s'$ (which includes the case $r=1$). Otherwise, if $p\ne p'$, \eqref{2exracondend''} follows from the second statement in \eqref{p>0=>p>p'}.

\subsubsection{The Proof - Case 1}\label{ss:p<0}  It follows from \Cref{ss:repcond} that, for $p\le 0$, \eqref{2exracond} is equivalent to  $r\le s'$. The case $r\le s'$ will be treated in Step 1 below while Step 2 will deal with the case $r>s'$.
We remark that Step 1 with $r=1$ will be used in the proof of Step 2.

\noindent {\bf Step 1.} Assume $r\le s'$.  We proceed by induction on $r$. The proofs of the cases $r=1$ and $r>1$ can be read independently. The argument for the latter is simpler and only evokes the simplicity of $V$ in the case $r=1$ with no partial use of its proof. 

\noindent {\bf Base of induction.} We will now treat the case $r=1$. However, we will prepare the argument without this restriction and set $r=1$ later on. We will also use a subinduction on $s$, which clearly starts for $s=0$, since $m-m'\notin\mathscr R_{i,j'}^{r,s'}$. Thus, we assume $s\ge 1$.   

If $V$ were reducible, so would be $V'$ and therefore, there would exist $\bs\lambda\in\mathcal P^+$ such that 
\begin{equation}\label{e:lambdanotop}
	\bs\lambda < \bs\pi\bs\omega_{i,a,r}.
\end{equation}
and nonzero homomorphisms
\begin{equation}\label{lambdainV'}
	L_q(\bs\lambda)\to V' \quad\text{and}\quad V\to L_q(\bs\lambda).
\end{equation}
Since $ L_q(\bs\omega_{j',b,s'})\otimes L_q(\bs\omega_{i,a,r})$ is simple, \eqref{e:toplow} implies  there exists $\bs\nu\in\mathcal P^+$ such that
\begin{equation}
	\bs\lambda = \bs\nu\bs\omega_{j',aq^{m-m'},s'} \qquad\text{and}\qquad L_q(\bs\nu)\hookrightarrow T_{r,s}^m:=L_q(\bs\omega_{j,aq^m,s})\otimes  L_q(\bs\omega_{i,a,r}).
\end{equation}
In fact, $L_q(\bs\nu)$ is the socle of $T_{r,s}^m$ and, since $m>0$, we have $\bs\nu < \bs\omega_{i,a,r}\bs\omega_{j,aq^m,s}$.

Setting $\widetilde{\bs\pi}=\bs\pi(\bs\omega_{j,aq^{m+s-1}})^{-1} = \bs\omega_{j,aq^{m-1},s-1}\bs\omega_{j',aq^{m-m'},s'},$ note also that we have an inclusion 
\begin{equation}\label{e:tilpitpj}
	L_q(\bs\pi) \hookrightarrow  L_q(\widetilde{\bs\pi})\otimes L_q(\bs\omega_{j,aq^{m+s-1}}),
\end{equation}
which can be combined with the above one to obtain
\begin{equation*}
	L_q(\bs\lambda)\hookrightarrow    L_q(\widetilde{\bs\pi})\otimes T_{r,1}^{m+s-1}.
\end{equation*}
By \eqref{e:frobrec}, this gives rise to a nonzero map
\begin{equation*}
	L_q(\widetilde{\bs\pi})^*\otimes L_q(\bs\lambda)\to T_{r,1}^{m+s-1}.
\end{equation*}
Let $L_q(\bs\mu)$ be a simple quotient of the image of this map. Thus, we get an epimorphism
\begin{equation*}
	L_q(\widetilde{\bs\pi})^*\otimes L_q(\bs\lambda)\to L_q(\bs\mu)
\end{equation*}
and, using \eqref{e:frobrec} again, an inclusion
\begin{equation}\label{e:linm}
	L_q(\bs\lambda)\hookrightarrow T':=L_q(\widetilde{\bs\pi})\otimes L_q(\bs\mu).
\end{equation}
Since $L_q(\bs\mu)$ is a simple factor of $T_{r,1}^{m+s-1}$, a description of the simple factors of $T_{r,1}^{m+s-1}$, which is done in \Cref{soctpkr} for $r=1$, allows us to study the possibilities for $\bs\mu$ to reach a contradiction. 

For instance, if $\bs\mu=\bs\omega_{j,aq^{m+s-1}}\bs\omega_{i,a,r}$, we claim
\begin{equation*}
	L_q(\bs\mu) \otimes L_q(\widetilde{\bs\pi}) \quad\text{is highest-$\ell$-weight.}
\end{equation*}
Using this and \Cref{l:subsandq}, it follows from \eqref{e:linm} that
\begin{equation}\label{e:linm'}
	\bs\lambda = \bs\mu\widetilde{\bs\pi} = \bs\pi\bs\omega_{i,a,r},
\end{equation}
contradicting \eqref{e:lambdanotop}. Thus, we can assume 
\begin{equation}\label{smallmu}
	\bs\mu < \bs\omega_{j,aq^{m+s-1}}\bs\omega_{i,a,r}.
\end{equation}
To prove the claim, note 
\begin{equation*}
	m+s-1\ge m-1\ge m-m'.
\end{equation*}
Thus, by \Cref{l:hlwquot}, we are left to check that
\begin{gather*}
	\tilde V:= L_q(\bs\omega_{i,a,r})\otimes  L_q(\bs{\tilde\pi})  \quad\text{is highest-$\ell$-weight.}
\end{gather*}
We proceed by induction on $s$ to show it is actually simple. The case $s=1$ is immediate from the assumption $|m-m'|\notin \mathscr R_{i,j'}^{r,s'}$. Otherwise, wee need to check whether $\tilde V$ satisfies all the assumptions we have made on $V$ or not, starting by checking that its $q$-factorization graph is as in (2) of \Cref{ss:tpfund}. We have replaced $m$ by $m-1$ and $s$ by $s-1$, so the condition required on the graph rephrases as
\begin{equation*}
	m-1\in\mathscr R_{i,j}^{r,s-1}, \qquad (m-1)-(m-m')\in\mathscr R_{j,j'}^{s-1,s'}, \quad\text{and}\qquad |m-m'|\notin\mathscr R_{i,j'}^{r,s'}.
\end{equation*}
Evidently, the last of these is satisfied. Note
\begin{equation*}
	m-1 = r+(s-1) + d(i,j) - 2p
\end{equation*}
and, since we are assuming $-d([i,j],\partial I)\le p\le 0$, the first condition is verified. The checking of the second is similar since we know $p'\le 0$ by \eqref{p'<ifp<}. The inductive argument is completed by checking that $\tilde V$ satisfies the corresponding version of \eqref{2gencond}, since  \eqref{2exracond'} is obviously satisfied. All conditions are easily checked after noting that $J$ is also the minimal subdiagram satisfying
\begin{equation*}
	i,j\in J \qquad\text{and}\qquad m-1\in\mathscr R_{i,j,J}^{r,s-1}.
\end{equation*}
Indeed, if $J'\subsetneqq J$, we have
\begin{equation*}
	d([i,j],\partial J')< d([i,j],\partial J) \stackrel{\eqref{J<>p}}{=} -p.
\end{equation*}
But then,
\begin{equation*}
	m-1 = r+(s-1) + d(i,j) -2p \notin \mathscr R_{i,j,J'}^{r,s-1}.
\end{equation*}
This completes the proof that we can assume \eqref{smallmu}.

Note \eqref{smallmu} implies $T_{r,1}^{m+s-1}$ is reducible and, hence,
\begin{equation}\label{newpforsmallmu}
	m+s-1 = r+1 + d(i,j) - 2\left(p-s +1\right) \in\mathscr R_{i,j}^{r,1}.
\end{equation}
We now set $r=1$. To simplify the use of the formulas from  \Cref{soctpkr},  we assume, without loss of generality, that $i\ge j$.  Recall $T_{1,1}^{m+s-1}$ has length $2$ by \eqref{soctpd} and \eqref{socforl2}. Using  \eqref{newpforsmallmu} in \eqref{socforl2}, we have
\begin{equation}
	\bs\mu=\bs\omega_{j+p-s,aq^{s-p+d(i,j)}}\ \bs\omega_{i+s-p,aq^{s-p}}.
\end{equation}
We claim that 
\begin{equation}\label{T'=lambda}
	T'\cong L_q(\bs\lambda).
\end{equation}
In particular, the first equality in \eqref{e:linm'} remains valid and  \eqref{lambdainV'} implies there exists  a monomorphism $T'\to V'$. Let $K$ be the minimal connected subdiagram such that $m+s-1\in\mathscr R_{i,j,K}^{1,1}$. It follows from \eqref{newpforsmallmu} and \eqref{e:tpfroots=min} that
\begin{equation*}
	K = [j+p-s+1,i-p+s-1] \quad\text{and, hence,}\quad \bs\mu_K=\bs 1.
\end{equation*}
Moreover, it follows from  the first equality in \eqref{e:linm'}, \eqref{e:tpfroots},  and \eqref{e:tpfroots=min} that
\begin{align*}
	\omega_i + \wt(\bs\pi) - \wt(\bs\lambda) & = \omega_i + \wt(\bs\pi) - (\wt(\bs\mu)+\wt(\widetilde{\bs\pi})) = \omega_i + (\wt(\bs\pi) - \wt(\widetilde{\bs\pi}))-\wt(\bs\mu) \\
	& = \omega_i + \omega_j - \wt(\bs\mu)\in Q_K^+.
\end{align*}
Therefore,  \Cref{p:subJrest} implies we also have a monomorphism 
\begin{equation*}
	L_q(\widetilde{\bs\pi}_K)\to  L_q(\bs\pi_K)\otimes L_q((\bs\omega_{i,a})_K)
\end{equation*}
An application of \eqref{e:frobrec} for the subalgebra $U_q(\tlie g)_K$ implies we have an epimorphism
\begin{equation*}
	L_q(\widetilde{\bs\pi}_K)\otimes L_q((\bs\omega_{i,a})_K)^*\to  L_q(\bs\pi_K)
\end{equation*}
Since $h^\vee_K = (i-p+s-1)-(j+p-s+1)+2 = m+s-1$ and $w_0^K(i)=j$, this is the same as
\begin{equation*}
	L_q(\widetilde{\bs\pi}_K)\otimes L_q((\bs\omega_{j,aq^{m+s-1}})_K)\to  L_q(\bs\pi_K)
\end{equation*}
Combining with \eqref{e:tilpitpj}, it follows that this map is an isomorphism. However, let us see that
\begin{equation*}
	L_q(\widetilde{\bs\pi}_j)\otimes L_q((\bs\omega_{j,aq^{m+s-1}})_j)\cong L_q((\bs\omega_{j',aq^{m-m'},s'})_j)\otimes L_q((\bs\omega_{j,aq^{m-1},s-1})_j)\otimes L_q((\bs\omega_{j,aq^{m+s-1}})_j)
\end{equation*}
which is reducible since $\bs\omega_{j,aq^{m-1},s-1}\bs\omega_{j,aq^{m+s-1}}=\bs\omega_{j,aq^m,s}$, reaching the desired contradiction. Indeed, if $j'\ne j$ the isomorphism is clear since $(\bs\omega_{j',aq^{m-m'},s'})_j=\bs 1$. Otherwise, if $j'=j$, the isomorphism follows because $\bs\omega_{j',aq^{m-m'},s'}$ and $\bs\omega_{j,aq^m,s}$ are $q$-factors of $\bs\pi$ or, equivalently, $m'\notin\mathscr R_{j,j}^{s,s'}$. Indeed, the isomorphism is not true if and only if $$m'-1 = (m-1)-(m-m')\in\mathscr R_{j,j}^{s-1,s'},$$
which implies $m'\in\mathscr R_{j,j}^{s,s'}$. 

It remains to prove \eqref{T'=lambda}. By \eqref{soctpd} and  \Cref{irred}, it suffices to show that
\begin{equation}
	\begin{aligned}
		& L_q(\bs\omega_{j',aq^{m-m'},s'}\bs\omega_{j,aq^{m-1},s-1})\otimes L_q(\bs\omega_{j+p-s,aq^{s-p+d(i,j)}})  \quad\text{and}\\ & L_q(\bs\omega_{j',aq^{m-m'},s'}\bs\omega_{j,aq^{m-1},s-1})\otimes L_q(\bs\omega_{i+s-p,aq^{s-p}}) \quad\text{are simple.}
	\end{aligned} 
\end{equation}
By \Cref{l:hlwquot}, this follows, respectively, if we check.
\begin{equation}\label{T'=lambda1}
	|m-m'-(s-p+d(i,j))|\notin\mathscr{R}_{j+p-s,j'}^{1,s'},\qquad |m-1-(s-p+d(i,j))|\notin\mathscr{R}_{j+p-s,j}^{1,s-1},
\end{equation}
\begin{equation}\label{T'=lambda2}
	|m-m'-(s-p)|\notin\mathscr{R}_{i+s-p,j'}^{1,s'}, \qquad\text{and}\qquad |m-1 - (s-p) |\notin\mathscr{R}_{i+s-p,j}^{1,s-1}.
\end{equation}
The second condition in each of these lines is vacuous if $s=1$.

We start with the first condition in \eqref{T'=lambda2}. If  $m-m'-(s-p)\ge 0$, then $m\ge m'$ and \eqref{defp=-} implies
\begin{align*}
	m-m' - (s-p) & =  1+s' + d(i,j') - 2p_+ -(s-p)\\
	& = 1+s' + d(i+s-p,j') -2(p_+-d^{i}_{i+s-p,j'}+s-p).
\end{align*}
Thus, it suffices to check that $p_+-d^{i}_{i+s-p,j'}+s-p\ge 1= \min\{1,s'\}$. Indeed, using \eqref{eq:cotd} we get
\begin{equation*}
	p_+-d^{i}_{i+s-p,j'}+s-p \ge p_+-d(i,i+s-p)+s-p=p_+
\end{equation*}
and \eqref{2lbp+-} completes the checking. If $m-m'-(s-p)< 0$,  we have\footnote{Here we are using $p_-$ to shorten notation without making any assumption on the sign of $m-m'$.}
\begin{align*}
	-(m-m' - (s-p)) & =   1+s' + d(i+s-p,j') -2(p_--d^{i}_{i+s-p,j'}).
\end{align*}
But,
\begin{align*}
	p_--d^{i}_{i+s-p,j'} & \stackrel{\eqref{eq:cotd}}{\ge} p_- - d(i,j') \stackrel{\eqref{defp=-}}{=} (1-p +p' + d_{j,j'}^i )- d(i,j')\\ & \stackrel{\eqref{eq:cotd}}{=} 1-p +p' - d^{j'}_{i,j}\stackrel{\eqref{weird}}{\ge} 1,
\end{align*}
completing the proof of the first condition in \eqref{T'=lambda2}. As for the second, note
\begin{equation*}
	m-1-s+p \stackrel{\eqref{e:2tp1Am}}{=} d(i,j)-p = 1+(s-1) + d(i+s-p,j) -2s.
\end{equation*}
The middle expression shows this number is nonnegative and we are done since $s\ge\min\{1,s-1\}$. 

Next, consider the first condition in \eqref{T'=lambda1}. By \eqref{e:2tp1Am} and \eqref{e:2tp1Am'} we have
\begin{align*}
	m-m'-(s-p+d(i,j)) & = 1-s-s' + 2p' - p  -d(j,j')\\
	& =  1+s' +d(j',j+p-s) - 2(s'-p'+1 + d)
\end{align*}
where $$2d=d(j,j')+(j'-j)+2(s-1)\ge 0.$$ Since $p'\le 0$ by \eqref{p'<ifp<}, it follows that $s'-p'+1 + d\ge 1$ and, hence, 
$$m-m'-(1-p+d(i,j)) \notin\mathscr{R}_{j+p-1,j'}^{1,s'}.$$ 
On the other hand,
\begin{align*}
	-(m-m'-(s-p+d(i,j))) & = s+s'-1 - 2p' + p  +d(j,j')\\
	& =  1+s' +d(j',j+p-s) - 2(p'-p+1-d)
\end{align*}
where $2d = d(j,j')+(j-j')\ge 0$. By \eqref{weird}, we have
\begin{equation*}
	p'-p+1-d \ge 1+ d^{j'}_{i,j}-d = \frac{1}{2}(d(i,j')+j'-i)\ge 1 =\min\{1,s'\},
\end{equation*}
thus completing the proof of the first condition in \eqref{T'=lambda1}. As for the second, 
\begin{equation*}
	m-1-(s-p+d(i,j)) \stackrel{\eqref{e:2tp1Am}}{=} -p = 1+(s-1) + d(j+p-s,j) -2s.
\end{equation*}
The middle expression shows this number is nonnegative and we are done since $s\ge\min\{1,s-1\}$.

\noindent {\bf Inductive Argument.} For $r>1$ we will show $V$ is highest-$\ell$-weight which completes the proof by \eqref{Vishw}.  To prove this, it suffices to show that  
\begin{equation*}
	L_q(\bs\omega_{i,aq,r-1})\otimes L_q(\bs\omega_{i,aq^{1-r}})\otimes L_q(\bs\pi) \quad\text{is highest-$\ell$-weight.}
\end{equation*}
Since the tensor product of the first two factors is highest-$\ell$-weight, \Cref{cyc} implies it suffices to show 
\begin{equation}\label{r>1hlw}
	L_q(\bs\omega_{i,aq^{1-r}})\otimes L_q(\bs\pi) \quad\text{and}\quad  L_q(\bs\omega_{i,aq,r-1})\otimes L_q(\bs\pi) \quad\text{are highest-$\ell$-weight.}
\end{equation}
The remainder of the argument will show that these tensor products are actually simple, thus completing Step 1.

By \Cref{l:hlwquot}, \eqref{defredset}, and \eqref{e:krhwtp}, the first claim in \eqref{r>1hlw} follows if
\begin{equation}\label{p<0indhlw}
	m-(1-r) \notin\mathscr R_{i,j}^{1,s} \qquad\text{and}\qquad (m-m')-(1-r)\notin\mathscr R_{i,j'}^{1,s'}.
\end{equation}
For the second number, note that, if $m-m'\le 0$, then
\begin{equation*}
	m-m'-(1-r)\le r-1 < r+1+d(i,j') = \min \mathscr R_{i,j'}^{1,s'}.
\end{equation*}
Thus, we can assume $m>m'$ in which case we have
\begin{equation}
	(m-m')-(1-r) \stackrel{}{=} r+s'+d(i,j') - 2p_+ - (1-r) = 1+s' +d(i,j') - 2(p_+ - r+1).
\end{equation}
Since $r\le s'$, \eqref{2lbp+-} implies $p_+ - r+1\ge 1=\min\{1,s'\}$, thus proving the second claim in \eqref{p<0indhlw}.
For the first claim in \eqref{p<0indhlw}, we have
\begin{equation*}
	0\le m-(1-r) = 1+s + d(i,j) - 2(p-r+1).
\end{equation*}
If $p-r+1<-d([i,j],\partial I)$, \eqref{p<0indhlw} is proved. Otherwise,
instead of checking \eqref{p<0indhlw}, we assume its first claim fails and use the induction hypothesis to prove directly that the first tensor product in \eqref{r>1hlw} is simple. 

Note we have $a$ replaced by $aq^{1-r}$ and $r$ replaced by $1$, so $m-(1-r)$ plays the role that $m$ did while $p-r+1$ plays the role that $p$ did. In particular, the initial assumption of \Cref{ss:p<0} is satisfied. Let us check whether the associated $q$-factorization graph  is as in (2) of \Cref{ss:tpfund}. i.e., if
\begin{equation}\label{remain2i1}
	m-(1-r)\in\mathscr R_{i,j}^{1,s}\qquad m'\in\mathscr R_{j,j'}^{s,s'},\qquad\text{and}\qquad |m-m'-(1-r)|\notin\mathscr R_{i,j'}^{1,s'}.
\end{equation}
The second condition is obvious and we have already partially proved the third in \eqref{p<0indhlw}, remaining to consider the case $|m-m'-(1-r)|=-(m-m'-(1-r))$. In this case we necessarily have $m\le m'$ and, thus,
\begin{equation*}
	-(m-m'-(1-r)) = r+s'+d(i,j')-2p_-+(1-r) = 1+s'+d(i,j')-2p_-
\end{equation*}
and the proof is completed using \eqref{2lbp+-}. After this, note that if the first  condition on \eqref{remain2i1} fails, the  first tensor product in \eqref{r>1hlw} is simple by \Cref{l:hlwquot}. In this case we do not need to use the induction hypothesis on $r$.

Otherwise, the inductive argument on $r$ follows if we check the corresponding version of \eqref{2gencond} is satisfied. More precisely,
\begin{equation}\label{2gencond'}
	j'\in J', \qquad m'\in{\mathscr{R}_{j,j',J'}^{s,s'}}, \qquad |m-(1-r)-m'-\check h_{J'}|\in\mathscr R_{j,j',J'}^{1,s'},
\end{equation}
where $J'$ is the minimal subdigram satisfying
\begin{equation*}
	i,j\in J' \quad\text{and}\quad m-(1-r)\in\mathscr R_{i,j, J'}^{1,s}
\end{equation*}
Since $p-r+1\le 0$, $J'$ is characterized by
\begin{equation*}
	d(i,\partial J') = d(j,\partial J') = - (p-r+1) \qquad\text{and}\qquad \check h_{J'} = d(i,j) - 2(p-r).
\end{equation*}
It then follows from \eqref{J<>p} that $J\subseteq J'$ and this immediately implies the first two conditions in \eqref{2gencond'}.
Finally, using \eqref{e:2tp1Am}, \eqref{e:2tp1Am'}, and \eqref{J<>p}, one easily checks that
\begin{equation*}
	-(m-(1-r)-m'-\check h_{J'}) = 1+s' +d(j,j')-2p'.
\end{equation*}
Since $p'\le 0$ by \eqref{p'<ifp<}, the third condition in \eqref{2gencond'} follows if we check
\begin{equation*}
	p'\ge -d([j,j'],\partial J').
\end{equation*}
But this is immediate from  the second condition in \eqref{2gencond'}.

To prove the second claim in \eqref{r>1hlw}, we again use the induction hypothesis on $r$ to show the corresponding module is simple. We again start by checking whether the $q$-factorization graph is as in (2) of \Cref{ss:tpfund}, which this time rephrases as 
\begin{equation}
	m-1\in\mathscr R_{i,j}^{r-1,s}, \qquad m'\in\mathscr R_{j,j'}^{s,s'}, \qquad\text{and}\qquad |m-m'-1|\notin \mathscr R_{i,j'}^{r-1,s'}.
\end{equation}
The second condition is obvious and, since
\begin{equation*}
	m-1 = (r-1)+s + d(i,j)-2p
\end{equation*}
and we are assuming $p\le 0$, the first condition and the initial assumption of \Cref{ss:p<0}  are also obvious. For the third, if $m> m'$, we have
\begin{equation*}
	0\le m-m'-1 = (r-1)+s'+d(i,j') - 2p_+ \qquad\text{and}\qquad p_+\ge r>r-1=\min\{r-1,s'\}
\end{equation*}
where we used \eqref{2lbp+-} and the hypothesis $r\le s'$ to obtain the inequalities. If $m\le m'$, then
\begin{equation*}
	0< -(m-m'-1) = (r-1)+s'+d(i,j') - 2(p_--1).
\end{equation*}
Now,
\begin{equation*}
	p_--1 \stackrel{\eqref{defp=-}}{=} r-p +p' + d_{j,j'}^i-1 \stackrel{\eqref{weird}}{\ge} r + d^{j'}_{i,j} + d_{j,j'}^i-1 \ge r-1=\min\{r-1,s'\}.
\end{equation*}

To complete the inductive argument, since $r-1<s'$ and, hence, we remain under the assumption of Step 2, it remains to check the corresponding version of \eqref{2gencond}, which reads
\begin{equation}\label{2gencond''}
	j'\in J', \qquad m'\in{\mathscr{R}_{j,j',J'}^{s,s'}}, \qquad |m-1-m'-\check h_{J'}|\in\mathscr R_{j,j',J'}^{r-1,s'},
\end{equation}
where $J'$ is the minimal subdigram satisfying
\begin{equation*}
	i,j\in J' \quad\text{and}\quad m-1\in\mathscr R_{i,j, J'}^{r-1,s}
\end{equation*}
It is immediate that $J'=J$ and, therefore, it remains to check the third condition in \eqref{2gencond''}. 
Using \eqref{e:2tp1Am}, \eqref{e:2tp1Am'}, and \eqref{J<>p}, one easily checks that 
\begin{equation*}
	-(m-1-m'-\check h_{J}) = (r-1)+s'+d(j,j') - 2(p'+r-2)
\end{equation*}
and, therefore, we need to check
\begin{equation*}
	-d([j,j'],\partial J) \le p'+r-2 \le r-1=\min\{r-1,s'\}.
\end{equation*}
The first inequality is clear from \eqref{weird} since $r>1$, while the second is clear from \eqref{p'<ifp<}.

\noindent {\bf Step 2.} We now  show   $V$ is not simple if $r>s'$ or, equivalently, if
\begin{equation}\label{2exracond''}
	m-m'+1\in\mathscr{R}_{i,j',J}^{r-1,s'}.
\end{equation}
By \Cref{c:sJs}, it suffices show that $V_J$ is reducible while \eqref{2exracond''} implies
\begin{equation*}
	L_q((\bs\omega_{j',aq^{m-m'},s'})_J)\otimes L_q((\bs\omega_{i,aq^{-1},r-1})_J)
\end{equation*}
is reducible and highest-$\ell$-weight. \Cref{p:critredp} and \Cref{l:hlwquot} then imply
\begin{equation*}
	L_1:=L_q(\bs\pi_J)\otimes L_q((\bs\omega_{i,aq^{-1},r-1})_J) 
\end{equation*}
is also reducible and highest-$\ell$-weight. On the other hand, obviously so is
\begin{equation*}
	L_2: = L_q((\bs\omega_{i,aq^{r-1}})_J)  \otimes L_q((\bs\omega_{i,aq^{-1},r-1})_J).
\end{equation*}
By \Cref{p:critred2p}, we are done if we show
\begin{equation}\label{Nissimple}
	N: = L_q((\bs\omega_{i,aq^{r-1}})_J) \otimes  L_q(\bs\pi_J)^* \quad\text{is simple.}
\end{equation}
To do this, we will check the $q$-factorization graph of $N$ is  as in (2) of \Cref{ss:tpfund} and that the corresponding version of \eqref{2gencond} holds. Then, \eqref{Nissimple} follows from the $r=1$ case which has been shown already. 

Set
\begin{equation*}
	k=w_0^J(j') \qquad\text{and}\quad L_q(\bs\pi') \cong L_q(\bs\pi_J)^*. 
\end{equation*}
We start by checking that the  $q$-factorization graph of $N$ is
\begin{equation}\label{newgraph}
	\begin{tikzcd}
		\stackrel{1}{i} \arrow[r,"\tilde m"] & \stackrel{s'}{k} & \arrow[swap,l,"m'"] \stackrel{s}{i}
	\end{tikzcd} \quad \text{where}\quad \tilde m= r-1 - m +m'+\check h_J.
\end{equation}
Recall $i=w_0^J(j)$, and, therefore,
\begin{equation}\label{dualdist}
	d(i,k)=d(j,j'), \qquad d([j,j'],\partial J) = d([i,k],\partial J), 
\end{equation}
and
\begin{equation*}
	\bs\pi' = \bs\omega_{i,aq^{m-\check h_J},s}\ \bs\omega_{k,aq^{m-m'-\check h_J},s'}.
\end{equation*}
Evidently,
\begin{equation*}
	m' = (m-\check h_J) - (m-m'-\check h_J)\in\mathscr R_{i,k}^{s,s'}. 
\end{equation*}
Thus,  \eqref{newgraph} follows if we check 
\begin{equation*}
	|(m-\check h_J) - (r-1)|\notin \mathscr R_{i,i}^{1,s} \qquad\text{and}\qquad \tilde m \in \mathscr R_{i,k}^{1,s'}.
\end{equation*}
The first of these follows since 
\begin{equation*}
	0\le  1+s+d(i,i)-2\cdot 1=s-1 = (m-\check h_J) - (r-1) \quad\text{and}\quad 1\ge \min\{1,s\},
\end{equation*}
where we used \eqref{e:2tp1Am} and \eqref{J<>p} in the second equality. As for the second, one easily checks that
\begin{equation}\label{newm}
	\tilde m = (r-1)- (m-m'-\check h_J)  = 1+s'+d(j,j') - 2p' = 1+s'+d(i,k) - 2p' 
\end{equation}
and we know
\begin{equation*}
	- d([i,k],\partial J)=- d([j,j'],\partial J)\le p' \stackrel{\eqref{p'<ifp<}}{\le } 0.
\end{equation*}

Now, let us write down the corresponding version of \eqref{2gencond}. From \eqref{newm}, the minimal subdiagram $J'$ such that
\begin{equation*}
	i,k\in J' \quad\text{and}\quad \tilde m\in\mathscr R_{i,k, J'}^{1,s'}
\end{equation*}
is characterized by the property
\begin{equation*}
	d(i,\partial J')= d(k,\partial J') = -p' \qquad\text{and}\qquad \check h_{J'} = d(i,k) - 2p'+2.
\end{equation*}
Thus,  the corresponding version of \eqref{2gencond} reads
\begin{equation}\label{2gencondstep3}
	i\in J', \qquad m'\in{\mathscr{R}_{i,k,J'}^{s,s'}}, \qquad |\tilde m-m'- \check h_{J'}|\in\mathscr R_{k,i,J'}^{1,s}.
\end{equation}
The first condition is obvious and the second follows from the characterization of $J'$, \eqref{dualdist}, and \eqref{e:2tp1Am'}. Finally,  using \eqref{newm} as well, we see that
\begin{equation*}
	m'+\check h_{J'} - \tilde m = 1+s+ d(i,k),
\end{equation*}
which proves the third condition in \eqref{2gencondstep3}.

\subsubsection{The Proof - Case 2}\label{ss:p>0}

In light of \eqref{Vishw} and \eqref{2exracondp>0}, we need to show that $V$ is highest-$\ell$-weight if and only if $r\le s'$ or $r>s'$ and $p\ne p'$.  In Steps 1 and 2 below, we prove that $V$ is highest-$\ell$-weight if $r\le s'$ or $r>s'$ and $p\ne p'$, respectively. In Step 3, we show $V$ is not highest-$\ell$-weight if $r>s'$ and $p=p'$. The main tool in all the steps is \Cref{cyc}, although previously proved cases are used as well.

\noindent {\bf Step 1.} Assume $r\le s'$. We shall use \Cref{cyc} several times in order to show $V$ is highest-$\ell$-weight, which completes the proof by \eqref{Vishw}. Let 
\begin{equation*}
	\bs{\tilde\pi} = \bs\pi\bs\omega^{-1}_{j',aq^{m-m'+s'-p'},p'} = \bs\omega_{j,aq^m,s}\bs\omega_{j',aq^{m-m'-p'},s'-p'}.
\end{equation*}
To conclude that $V$ is highest-$\ell$-weight, we will show that the following tensor products are highest-$\ell$-weight:
\begin{equation}\label{pishorttopi}
	L_q(\bs\omega_{j',aq^{m-m'+s'-p'},p'}) \otimes L_q(\bs{\tilde\pi}),
\end{equation}
\begin{equation}\label{ij'rest}
	L_q(\bs\omega_{i,a,r})\otimes L_q(\bs\omega_{j',aq^{m-m'+s'-p'},p'}),
\end{equation}
and
\begin{equation}\label{ipishort}
	L_q(\bs\omega_{i,a,r})\otimes  L_q(\bs{\tilde\pi}).
\end{equation}
Together with \Cref{cyc}, this implies
\begin{equation*}
	L_q(\bs\omega_{i,a,r})\otimes L_q(\bs\omega_{j',aq^{m-m'+s'-p'},p'}) \otimes L_q(\bs{\tilde\pi})
\end{equation*}
is also highest-$\ell$-weight. Moreover, \eqref{pishorttopi} implies we have a surjective map
\begin{equation*}
	L_q(\bs\omega_{i,a,r})\otimes L_q(\bs\omega_{j',aq^{m-m'+s'-p'},p'}) \otimes L_q(\bs{\tilde\pi}) \to V,
\end{equation*}
showing $V$ is highest-$\ell$-weight.

To prove \eqref{pishorttopi} is highest-$\ell$-weight  we use Lemma \ref{l:hlwquot} which says it suffices to show 
\begin{equation*}
	L_q(\bs\omega_{j',aq^{m-m'+s'-p'},p'}) \otimes L_q(\bs\omega_{j',aq^{m-m'-p'},s'-p'}) \quad\text{and}\quad L_q(\bs\omega_{j',aq^{m-m'+s'-p'},p'})\otimes L_q(\bs\omega_{j,aq^m,s})
\end{equation*}
are highest-$\ell$-weight. The first is clear since $s'-p'>0$. Proving the second is equivalent to checking that
\begin{equation*}
	m-(m-m'+s'-p') \notin\mathscr R_{j,j'}^{s,p'}
\end{equation*}
if $p'>0$. By \eqref{e:2tp1Am'} this number is equal to
\begin{equation*}
	s+d(j,j')-p' = s+p' + d(j,j')-2p'
\end{equation*}
and $p'\ge \min\{s,p'\}=p'$.

Showing that \eqref{ij'rest} is highest-$\ell$-weight is equivalent to checking that
\begin{equation*}
	m-m'+s'-p'\notin \mathscr R_{i,j'}^{r,p'}
\end{equation*}
if $p'>0$. Using \eqref{e:2tp1Am} and \eqref{e:2tp1Am'}, we see this number is equal to
\begin{equation*}
	r+d(i,j)-2p - d(j,j')+p' = r+p'+d(i,j') - 2(p+d^{j'}_{i,j}).
\end{equation*}
Since \eqref{p>0=>p>p'} implies $p+d^{j'}_{i,j}\ge p'\ge \min\{r,p'\}$, the checking is completed.

To prove \eqref{ipishort} is highest-$\ell$-weight, we will show
\begin{equation}\label{ipishort'}
	L_q(\bs\omega_{i,aq^{r-p},p})\otimes L_q(\bs{\tilde\pi}) \qquad\text{and}\qquad L_q(\bs\omega_{i,aq^{-p},r-p})\otimes L_q(\bs{\tilde\pi})
\end{equation}
are also  highest-$\ell$-weight. Assuming this and noting that 
\begin{equation*}
	L_q(\bs\omega_{i,aq^{r-p},p})\otimes L_q(\bs\omega_{i,aq^{-p},r-p})
\end{equation*}
is highest-$\ell$-weight and maps onto $L_q(\bs\omega_{i,a,r})$, \Cref{cyc} implies
\begin{equation*}
	L_q(\bs\omega_{i,aq^{r-p},p})\otimes L_q(\bs\omega_{i,aq^{-p},r-p}) \otimes L_q(\bs{\tilde\pi})
\end{equation*}
is highest-$\ell$-weight and maps onto \eqref{ipishort}, thus completing the proof.

To prove the first tensor product in \eqref{ipishort'} is highest-$\ell$-weight  we use Lemma \ref{l:hlwquot} which says it suffices to show 
\begin{equation*}
	L_q(\bs\omega_{i,aq^{r-p},p}) \otimes  L_q(\bs\omega_{j,aq^m,s}) \quad\text{and}\quad L_q(\bs\omega_{i,aq^{r-p},p})\otimes L_q(\bs\omega_{j',aq^{m-m'-p'},s'-p'})
\end{equation*}
are highest-$\ell$-weight. For the first, note
\begin{equation*}
	m-(r-p) = s-p+d(i,j) = p+s+d(i,j)-2p
\end{equation*}
and $p\ge\min\{p,s\}=p$, showing $m-(r-p)\notin \mathscr R_{i,j}^{p,s}$ as desired.
As for the second, we have
\begin{align*}
	m-m'-p' -(r-p)  &= -s'   + d(i,j)-d(j',j) +p'-p\\ & \stackrel{\eqref{p>0=>p'>0}}{=} p+(s'-p')+d(i,j') - 2(p+s'-p').
\end{align*}
Since $p+s'-p'\ge\min\{p,s'-p'\}$, this number is not in $\mathscr R_{i,j'}^{p,s'-p'}$, as desired.

It remains to prove that the second  tensor product in \eqref{ipishort'} is highest-$\ell$-weight. We will actually prove it is simple by showing that it is a module in the conditions studied in Step 2 of \Cref{ss:p<0}. We begin by checking that its $q$-factorization graph is
\begin{equation*}
	\begin{tikzcd}
		\stackrel{r-p}{i} & \arrow[swap,l,"\tilde m"]  \stackrel{s}{j} \arrow[r,"\tilde m'"] & \stackrel{s'-p}{j'} \qquad \text{where}\qquad \tilde m = m+p \quad\text{and}\quad \tilde m'=m'+p'.
	\end{tikzcd} 
\end{equation*}
Or, equivalently,
\begin{equation}\label{newgraph'}
	\tilde m\in \mathscr R_{i,j}^{r-p,s}, \qquad \tilde m'\in\mathscr R_{j,j'}^{s,s'-p'}, \qquad\text{and}\qquad |\tilde m-\tilde m'|\notin \mathscr R_{i,j'}^{r-p,s'-p'}.
\end{equation}
The first two are clear since
\begin{equation*}
	\tilde m = (r-p)+s+d(i,j)-2\tilde p, \quad\tilde m' = s+(s'-p')+d(j,j')-2\tilde p', \quad\text{where}\qquad \tilde p=\tilde p'=0.
\end{equation*}
For the third, note
\begin{equation*}
	\tilde m-\tilde m' = m-m'+p-p'.
\end{equation*}
Therefore, by \eqref{p>0=>p>p'}, if $\tilde m\le \tilde m'$, we also have $m\le m'$ and we get
\begin{align*}
	|\tilde m- \tilde m'| & = -(\tilde m- \tilde m') = -(m-m')-p+p' \stackrel{\eqref{defp=-}}{=} r+s'+d(i,j')-2p_- -p+p'\\
	& = (r-p) + (s'-p') + d(i,j')-2(p_- - p').
\end{align*}
Using \eqref{2lbp+-} and the assumption $r\le s'$, we get
\begin{equation*}
	p_- - p' \ge \min\{r,s'\} - p' = r-p' \stackrel{\eqref{p>0=>p>p'}}{\ge} r-p \ge \min\{r-p,s'-p'\},
\end{equation*}
which completes the proof of \eqref{newgraph'} in this case. If $\tilde m>\tilde m'$, then using \eqref{e:2tp1Am} and \eqref{e:2tp1Am'} we get
\begin{align*}
	|\tilde m- \tilde m'| & = \tilde m- \tilde m' = m-m'+p-p' = (r-p) + (s'-p') + d(i,j')-2(s'-p'+d^{j'}_{i,j}).
\end{align*}
Since
\begin{equation*}
	s'-p'+d^{j'}_{i,j} \ge s'-p' \ge \min\{r-p,s'-p'\},
\end{equation*}
the proof of \eqref{newgraph'} is complete.

The remaining conditions required in Step 2 of \Cref{ss:p<0} are now easily checked. Indeed, since we have already seen that $\tilde p=\tilde p'=0$, it remains to check that $r-p\le s'-p'$ which is clear from \eqref{p>0=>p>p'} and the assumption $r\le s'$. Thus, the second  tensor product in \eqref{ipishort'} is simple as claimed.

\noindent {\bf Step 2.} We now show that $V$ is highest-$\ell$-weight if $r>s'$ and $p> p'$. Note this follows if we show
\begin{equation*}
	W:=  L_q(\bs\omega_{i,aq^{r-p+p'},p-p'})\otimes L_q(\bs\omega_{i,aq^{p'-p},r-(p-p')})\otimes  L_q(\bs\pi) \qquad\text{is highest-$\ell$-weight.}
\end{equation*}
Indeed, since $r-p+p'>p'-p$ (because $p>p'$), the tensor product of the first two factors is highest-$\ell$-weight and has $L_q(\bs\omega_{i,a,r})$ as irreducible quotient. Thus, if we show $W$ is highest-$\ell$-weight, the natural map $W\to V$ implies that so is $V$. Using \Cref{cyc}, it now suffices to show
\begin{equation}\label{Step3from2}
	L_q(\bs\omega_{i,aq^{r-p+p'},p-p'})\otimes  L_q(\bs\pi) \qquad\text{and}\qquad L_q(\bs\omega_{i,aq^{p'-p},r-(p-p')})\otimes  L_q(\bs\pi)
\end{equation}
are highest-$\ell$-weight.

For the first of these tensor products, \Cref{l:hlwquot} implies it suffices to check
\begin{equation*}
	m-(r-p+p')\notin\mathscr R_{i,j}^{p-p',s} \qquad\text{and}\qquad  (m-m')-(r-p+p')\notin\mathscr R_{i,j'}^{p-p',s'}.
\end{equation*}
Indeed,
\begin{equation*}
	m-(r-p+p') \stackrel{\eqref{e:2tp1Am}}{=} s+ d(i,j) - p -p' = (p-p')+s+d(i,j)-2p
\end{equation*}
and $p\ge p-p'\ge\min\{p-p',s\}$, while, using \eqref{e:2tp1Am'} as well, we get
\begin{align*}
	(m-m')-(r-p+p') & = d(i,j)-d(j,j') - p +p'-s'\\ & = (p-p')+s' + d(i,j') - 2(p-p'+s'+d^{j'}_{i,j})
\end{align*}
and $p-p'+s'+d^{j'}_{i,j}\ge \min\{p-p',s'\}$.

We will now check that the second tensor product in \eqref{Step3from2} satisfies the conditions either of Step 2 or of Step 2 from \Cref{ss:p<0}. This implies it is simple, thus completing the proof. Hence, we need to check that its $q$-factorization graph is
\begin{equation*}
	\begin{tikzcd}
		\stackrel{\tilde r}{i} & \arrow[swap,l,"\tilde m"]  \stackrel{s}{j} \arrow[r,"m'"] & \stackrel{s'}{j'} \qquad\text{where}\qquad \tilde m = m+p-p' \quad\text{and}\quad \tilde r=r-(p-p'),
	\end{tikzcd} 
\end{equation*}
or, equivalently,
\begin{equation}\label{newgraphStep3}
	\tilde m\in \mathscr R_{i,j}^{\tilde r,s} \qquad\text{and}\qquad |\tilde m- m'|\notin \mathscr R_{i,j'}^{\tilde r,s'}.
\end{equation}
The other requirement to satisfy either one of Steps 2 is
\begin{equation*}
	\tilde r\le s', 
\end{equation*}
which follows from \eqref{3rfP<0'}. 
For the first number in \eqref{newgraphStep3}, we have
\begin{equation*}
	\tilde m \stackrel{\eqref{e:2tp1Am}}{=} r+s+d(i,j)-p-p' = (r-(p-p')) + s + d(i,j) - 2p'.
\end{equation*}
We also know $0\le p'\le s'$ and $p<r$ which implies $p'< r-(p-p')$, as desired.\footnote{Note this last computation shows that we are in the conditions of Step 2 from \Cref{ss:p<0} if and only if $p'=0$.}
For the second, assume first $\tilde m\le m'$ which implies $m\le m'$ since $p'>p$. Then,
\begin{align*}
	|\tilde m-m'| = - (m-m') -(p-p') \stackrel{\eqref{defp=-}}{=} \tilde r + s' +d(i,j') - 2p_-
\end{align*}
and
\begin{equation*}
	p_- \stackrel{\eqref{2lbp+-}}{\ge} \min\{r,s'\}=s' \ge \min\{s',\tilde r\}.
\end{equation*}
On the other hand, if $\tilde m> m'$, \eqref{e:2tp1Am} and \eqref{e:2tp1Am'} imply
\begin{align*}
	|\tilde m-m'| & =m-m' + p-p' = r-s'+ d(i,j)-d(j,j')-p+p' \\ &= (r-(p-p'))+s'+d(i,j')-2(s'+d^{j'}_{i,j})
\end{align*}
and $s'+d^{j'}_{i,j}\ge s'\ge \min\{s',\tilde r\}$.

\noindent {\bf Step 3.} Finally, we show that $V$ is not highest-$\ell$-weight if $r>s'$ and $p=p'$.  We proceed by contradiction. More precisely, we will show using \Cref{cyc} that, if $V$ were highest-$\ell$-weight, so would be
\begin{equation*}
	T:=L_q(\bs\omega_{i,a,r})\otimes L_q(\bs\omega_{j,aq^{m-(s'-p)},s+s'-p }).
\end{equation*}
However, $m-(s'-p)\in\mathscr R_{i,j}^{r,s+s'-p}$ yielding a contradiction. Indeed,
\begin{equation*}
	m-(s'-p) \stackrel{\eqref{e:2tp1Am}}{=} r + (s+s'-p) + d(i,j) - 2s'
\end{equation*}
and $0\le s'<\min\{r,s+s'-p\}$ since $s'<r$ and $p<s$.

In order to show $T$ is highest-$\ell$-weight, it suffices to show this is true for
\begin{equation*}
	T\otimes L_q(\bs\omega_{j',aq^{m-m'},s'}).
\end{equation*}
Set
\begin{equation*}
	\bs{\tilde\pi} = \bs\omega_{j,aq^{m-(s'-p)},s+s'-p }\bs\omega_{j',aq^{m-m'},s'} = \bs\omega_{j,aq^{\tilde m},s'-p}\bs\pi, 
\end{equation*}
where
\begin{equation*}
	\tilde m = m - (s+s'-p).
\end{equation*}
We claim
\begin{equation}\label{step4s}
	L_q(\bs{\tilde\pi}) \cong L_q(\bs\omega_{j,aq^{m-(s'-p)},s+s'-p })\otimes L_q(\bs\omega_{j',aq^{m-m'},s'}),
\end{equation}
which implies
\begin{equation*}
	T\otimes L_q(\bs\omega_{j',aq^{m-m'},s'}) \cong L_q(\bs\omega_{i,a,r})\otimes L_q(\bs{\tilde\pi}),
\end{equation*}
and, hence, it suffices to show
\begin{equation*}
	\tilde V:= L_q(\bs\omega_{i,a,r})\otimes L_q(\bs{\tilde\pi}) \quad\text{is highest-$\ell$-weight.}
\end{equation*}
Finally, in order to show this, it suffices to show
\begin{equation}\label{Wishw}
	W \qquad\text{and}\qquad L_q(\bs\pi)\otimes L_q(\bs\omega_{j,aq^{\tilde m},s'-p}) \quad\text{are highest-$\ell$-weight}
\end{equation}
where 
\begin{equation*}
	W =L_q(\bs\omega_{i,a,r}) \otimes L_q(\bs\pi)\otimes L_q(\bs\omega_{j,aq^{\tilde m},s'-p}) = V\otimes  L_q(\bs\omega_{j,aq^{\tilde m},s'-p}).
\end{equation*} 
Indeed, if this is true, we obtain a natural epimorphism $ W\to \tilde V$, showing $\tilde V$ is highest-$\ell$-weight.

Note that, assuming $V$ were highest-$\ell$-weight, \Cref{cyc} implies the first claim in \eqref{Wishw} would follow from the second together with 
\begin{equation*}
	L_q(\bs\omega_{i,a,r})\otimes L_q(\bs\omega_{j,aq^{\tilde m},s'-p}) \quad\text{is highest-$\ell$-weight}.
\end{equation*}
To check the latter, note
\begin{equation*}
	\tilde m = r-s'-p+d(i,j) = r+(s'-p)+d(i,j)-2s'
\end{equation*}
and, since $r>s'$, $s'\ge s'-p=\min\{r,s'-p\}$, so $\tilde m\notin\mathscr R_{i,j}^{r,s'-p}$ as desire. As for the former,  \Cref{l:hlwquot} implies it suffices to check
\begin{equation*}
	\tilde m- m\notin\mathscr R_{j,j}^{s,s'-p} \qquad\text{and}\qquad \tilde m - (m-m')\notin\mathscr R_{j,j'}^{s'-p,s'}.
\end{equation*}
The first is clear since $\tilde m\le m$, while
\begin{align*}
	\tilde m - (m-m') & = m'-(s+s'-p) \stackrel{\eqref{e:2tp1Am'}}{=} d(j,j') - 2p' + p = d(j,j') - p\\
	& = (s'-p)+s'+d(j,j') -2s',
\end{align*}
and $s'\ge \min\{s',s'-p\}$.

It remains to check \eqref{step4s}, i.e.,
\begin{equation*}
	|m-(s'-p)-(m-m')| = |m'-s'+p|\notin\mathscr R_{j,j'}^{s+s'-p,s'}.
\end{equation*}
Indeed, $m'-s'+p=s+d(j,j')-p \ge 0$ since $p<s$, so
\begin{equation*}
	|m'-s'+p| = s+d(j,j')-p = (s+s'-p)+ s'+d(j,j')-2s', 
\end{equation*}
and $s'\ge \min\{s',s+s'-p\}=s'$.

\subsection{Reality of Trees}\label{ss:realtree} 
We start with:

\begin{proof}[Proof of \Cref{c:3aline}]
	Letting $\bs\pi = (\bs\omega_{i,a,r})^2\bs\omega_{j,aq^m,s}$, we see that $G(\bs\pi)$ is 
	\begin{equation*}
		\begin{tikzcd}
			\stackrel{r}{i} & \arrow[swap,l,"m"]  \stackrel{s}{j} \arrow[r,"m"] & \stackrel{r}{i} 
		\end{tikzcd}
	\end{equation*}
	Therefore, we need to check that \eqref{2gencond}  and \eqref{2exracond} hold under the given assumptions, i.e., $j'=i, s'=r, m'=m$. The former becomes
	\begin{equation*}
		\check h_J \in\mathscr R_{j,i,J}^{r,r} = \{2r + d(i,j)-2k: -d([i,j],\partial J)\le k< r \}.
	\end{equation*}
	If $p$ is given by \eqref{e:2tp1Am}, we have seen in \eqref{J<>p} and \eqref{Jp>0} that
	\begin{equation*}
		\check h_J = 2+d(i,j) - 2\min\{0,p\}.
	\end{equation*}
	Thus, letting $k = r-1 + \min\{0,p\}$, we have 
	\begin{equation*}
		-d([i,j],\partial J) \le \min\{0,p\} \le k <r \qquad\text{and}\qquad \check h_J = 2r+d(i,j)-2k,
	\end{equation*}
	which shows \eqref{2gencond} holds.
	
	On the other hand, \eqref{2exracond}  becomes
	\begin{equation*}
		1\notin\mathscr R_{i,i,J}^{r-1,r} =  \{2r-1 -2k: -d(i,\partial J)\le k< r-1 \} \quad\text{if}\quad r>1.
	\end{equation*}
	Since the unique solution of the equation $2r-1 -2k=1$ is $k=r-1$, we are done.
\end{proof}

Finally:

\begin{proof}[Proof of \Cref{t:realtree}]
	We proceed by induction on the number $N$ of $q$-factors of $\bs\pi$, which clearly starts if $N=1$. Thus, assume $N>1$ and recall it suffices to show $L_q(\bs\pi)\otimes L_q(\bs\pi)$ is highest-$\ell$-weight by \Cref{c:vnvstar}.
	
	Choose $\bs\omega_{i,a,r}\in\partial G$, let $\bs\omega_{j,aq^m,s}$ be the unique vertex adjacent to $\bs\omega_{i,a,r}$, and write 
	\begin{equation*}
		\bs{\pi}=\bs{\varpi}^{\prime}\bs{\omega}_{i,a,r}\bs{\omega}_{j,aq^m,s} \qquad\text{and}\qquad \bs\varpi = \bs{\varpi}^{\prime}\bs{\omega}_{j,aq^m,s}.
	\end{equation*}	
	In particular, $|m|\in\mathscr R_{i,j}^{r,s}$ and  $L_q(\bs\omega_{i,a,r})\otimes L_q(\bs\omega)$ is simple for every $q$-factor $\bs\omega$ of $\bs\varpi'$. \Cref{l:hlwquot} then implies $L_q(\bs{\omega}_{i,a,r})\otimes L_q(\bs{\varpi}^{\prime})$ is simple. 
	Note also that
	\begin{equation*}
		\bs\pi^2 = \bs\pi \bs{\omega}_{i,a,r} \bs\varpi
	\end{equation*}
	and, since $G$ is a tree, so is $G(\bs\varpi)$ by \Cref{l:elemtree}. Hence, the induction hypothesis applies to $\bs\varpi$.

	Without loss of generality, assume $m>0$, so that $L_q(\bs\omega_{i,a,r})\otimes L_q(\bs\omega_{j,aq^m,s})$ is highest-$\ell$-weight. We claim
	\begin{equation}\label{realtreehlf}
		L_q(\bs{\pi})\otimes L_q(\bs{\omega}_{i,a,r}), \qquad L_q(\bs{\pi})\otimes L_q(\bs{\varpi}), \qquad\text{and}\qquad L_q(\bs{\omega}_{i,a,r})\otimes L_q(\bs{\varpi})
	\end{equation}
	are highest-$\ell$-weight. Together with \Cref{cyc}, this implies 
	\begin{equation*}
		L_q(\bs{\pi})\otimes L_q(\bs{\omega}_{i,a,r}) \otimes L_q(\bs{\varpi})
	\end{equation*}
	is highest-$\ell$-weight and has $L_q(\bs\pi)\otimes L_q(\bs\pi)$ as a quotient, thus completing the proof.
	
	For the last tensor product in \eqref{realtreehlf}, it follows from the observations we have already made that $L_q(\bs\omega_{i,a,r})\otimes L_q(\bs\omega)$ is highest-$\ell$-weight for every $q$-factor $\bs\omega$ of $\bs\varpi$. Thus, the claim follows from \Cref{l:hlwquot} in this case. Moreover, since $L_q(\bs\varpi)$ is real by the induction hypothesis, this also implies that
	\begin{equation*}
		L_q(\bs{\omega}_{i,a,r})\otimes L_q(\bs{\varpi}) \otimes L_q(\bs{\varpi})
	\end{equation*}
	is highest-$\ell$-weight and has the middle tensor product in \eqref{realtreehlf} as a quotient. Thus, it remains to prove the claim for the first tensor product.
	
	To do that, let $\bs\varpi_+$ be the product of all $q$-factors $\bs\omega$ of $\bs\varpi'$ such that $\bs\omega \succ\bs\omega_{j,aq^m,s}$ and let
	$\bs\varpi_-$ be such that
	\begin{equation*}
		\bs\pi = \bs\varpi_+\bs{\omega}_{i,a,r}\bs\omega_{j,aq^m,s} \bs\varpi_-.
	\end{equation*}
	Consider
	\begin{equation*}
		W= L_q(\bs{\varpi}_+)\otimes L_q(\bs{\omega}_{i,a,r}\bs{\omega}_{j,aq^m,s}) \otimes L_q(\bs{\varpi}_-)\otimes L_q(\bs{\omega}_{i,a,r})
	\end{equation*}
	and note that, except for
	\begin{equation*}
		L_q(\bs{\omega}_{i,a,r}\bs{\omega}_{j,aq^m,s}) \otimes L_q(\bs{\omega}_{i,a,r}),
	\end{equation*}
	all the others $2$-fold ordered tensor products in the definition of $W$ are highest-$\ell$-weight by construction. On the other hand, this latter tensor product is simple by \Cref{c:3aline}. Therefore, $W$ is highest-$\ell$-weight and has the first tensor product in \eqref{realtreehlf} as a quotient. This completes the proof.
\end{proof}

\bibliographystyle{amsplain}

\end{document}